\documentclass[pdftex]{amsart} 

\usepackage{array,rotating,pdflscape}
\usepackage{enumerate,cool,mathtools}
\usepackage[margin=2.0cm]{geometry}
\usepackage{amsfonts,amssymb,amsmath,latexsym,upref}
\usepackage{verbatim} 
\usepackage[english]{babel}
\usepackage[latin1]{inputenc} 
\usepackage{color,tikz,pgf}

\usetikzlibrary{shapes, intersections, calc, patterns, positioning,
  matrix, arrows, fit, backgrounds, 
  decorations.pathmorphing, decorations.markings}

\newtheorem{lemma}[equation]{Lemma}
\newtheorem{prop}[equation]{Proposition}
\newtheorem{thm}[equation]{Theorem}
\newtheorem{cor}[equation]{Corollary}

\newtheorem{defn}[equation]{Definition}

\theoremstyle{definition}
\newtheorem{exmp}[equation]{Example}

\newtheorem{rmk}[equation]{Remark}

\numberwithin{equation}{section}

\newcommand{\Z}{\mathbf{Z}}

\newcommand{\Q}{\mathbf{Q}}

\newcommand{\F}{\mathbf{F}}

\newcommand{\Hom}[3][{}]{\operatorname{Hom}_{#1}(#2,#3)}

\newcommand{\ord}[2]{\operatorname{ord}_{#1}(#2)}

\newcommand{\GGL}[2]{{G}^{#1}_{#2}} 

\newcommand{\GL}[3][{}]{\operatorname{GL}^{#1}_{#2}(#3)}
\newcommand{\gl}{\GL n{\F_q}}
\newcommand{\Li}{\operatorname{L}}

\newcommand{\cat}[3]{\mathcal{#1}_{#2}^{#3}}  
  
\newcommand{\IM}[2]{\operatorname{IM}_{#1}(#2)}

\newcommand{\m}{morphism}

\newcommand{\gen}[1]{{\langle}#1{\rangle}}

\newcommand{\syl}[1]{Sylow $#1$-subgroup}

\newcommand{\pol}{polynomial}
\newcommand{\Euc}{Euler characteristic}
\newcommand{\rchi}{\widetilde{\chi}}

\newcommand{\mynote}[1]{\noindent{\textcolor{red}{\textbf{[#1]}}}}

\newcommand{\wh}[1]{\widehat{#1}}

\title{Equivariant Euler characteristics of subspace posets}

\author{Jesper M.~M\o ller}
\address{Institut for Matematiske Fag\\
  Universitetsparken 5\\
  DK--2100 K\o benhavn}
\email{moller@math.ku.dk}
\urladdr{htpp://www.math.ku.dk/~moller}

\thanks{Supported by the Danish National Research Foundation through
  the Centre for Symmetry and Deformation (DNRF92)}

\usepackage[bookmarks=true,bookmarksopen=false]{hyperref}
\hypersetup{
   pdftitle = {},
   pdfauthor = {Jesper Michael MÃÂ¸ller},
   pdfpagemode = {UseOutlines},
   pdfstartview = {FitH},
   pdfborder = {0 0 0},
   backref = {true},
   colorlinks = {true},
   urlcolor = {blue},
   citecolor = {blue},
   linkcolor = {blue},
   pdftoolbar = {true},
}


\begin{document}
\date{\today}
\begin{abstract}
  The ($p$-primary) equivariant reduced \Euc s of the building for
  the general linear group over a finite field are determined.
\end{abstract}
\maketitle
\tableofcontents

\section{Introduction}
\label{sec:intro}


Let $G$ be a finite group, $\Pi$ a finite $G$-poset, and $r \geq 1$ a
natural number.  The \emph{$r$th equivariant reduced \Euc\/} of the
$G$-poset $\Pi$ as defined by Atiyah and Segal \cite{atiyah&segal89}
and Tamanoi \cite{tamanoi2001}
is the normalized sum
\begin{equation}\label{defn:rchiPiG}
  \rchi_r(\Pi,G) = \frac{1}{|G|} \sum_{X \in \Hom{\Z^r}G} \rchi(C_{\Pi}(X(\Z^r))
\end{equation}
of the reduced \Euc s of the $X(\Z^r)$-fixed $\Pi$-subposets,
$C_{\Pi}(X(\Z^r))$, as $X$ ranges over the set of all homo\m s of
$\Z^r$ to $G$. For example, when $G$ acts trivially on $\Pi$,
$\rchi_r(\Pi,G) = \rchi(\Pi) |\Hom{\Z^r}G|/|G|$ where
$|\Hom{\Z^r}G|/|G|= |\Hom{\Z^{r-1}}{G}/G|$ is the number of conjugacy classes
of commuting $(r-1)$-tuples of elements of $G$ \cite[Lemma
4.13]{HKR2000}.  In this article we specialize to posets of linear
subspaces of finite vector spaces.  Let $q$ be a prime power,
$n \geq 1$ a natural number, $V_n(\F_q)$ the $n$-dimensional vector
space over $\F_q$, ${\Li}_n(\F_q)$ the $\gl$-lattice of subspaces of
$V_n(\F_q)$, and ${\Li}_n^*(\F_q) = {\Li}_n(\F_q) - \{0,V_n(\F_q)\}$
the proper part of ${\Li}_n(\F_q)$ consisting of nontrivial and proper
subspaces.  The general Definition~\eqref{defn:rchiPiG} takes the
following form in this context:

\begin{defn}\label{defn:chir}
  The $r$th, $r \geq 1$, equivariant reduced \Euc\ of the $\gl$-poset
  ${\Li}_n^*(\F_q)$ is the normalized sum
\begin{equation*}
  \rchi_r({\Li}_n^*(\F_q),\GL n{\F_q}) = \frac{1}{|\GL n{\F_q}|}
   \sum_{X \in \Hom {\Z^r}{\GL n{\F_q}}} \rchi(C_{\Li_n^*(\F_q)}(X(\Z^r))) 
\end{equation*}
of the \Euc s of the subposets $C_{\Li_n^*(\F_q)}(X(\Z^r))$ of
$X(\Z^r)$-invariant subspaces as $X$ ranges over all homo\m s of the
free abelian group $\Z^r$ on $r$ generators into the general linear
group $\gl$.
\end{defn}

The generating function for the sequence
$\rchi_r({\Li}_n^*(\F_q),\GL n{\F_q})$, $n \geq 1$, or $r$th
generating function for short, is the power series
\begin{equation}\label{eq:genfct}
  F_r(x,q) = 1+\sum_{n \geq 1} \rchi_r({\Li}_n^*(\F_q),\gl)x^n \in
  \Z[q][[x]], \qquad r \geq 1
\end{equation}
with coefficients in integral polynomials in $q$.


\begin{thm}\label{thm:chirgl} 
  $\displaystyle  F_{r+1}(x,q) = \prod_{0 \leq j \leq r}
  (1-q^jx)^{(-1)^{r-j}\binom{r}{j}}$ for all $r \geq 0$.
\end{thm}

 The first  generating functions $F_r(x,q)$ for $1 \leq r \leq 5$ are
 \begin{equation*} 
    1-x,  \quad   \frac{1-qx}{1-x}, \quad
     \frac{(1-x) (1-q^2x)}{(1-qx)^2}, \quad
    \frac{(1-qx)^3 (1-q^3x)}{(1-x)(1-q^2x)^3}, \quad
     \frac{(1-x)(1-q^2x)^6 (1-q^4x)}{(1-qx)^4(1-q^3x)^4}
   \end{equation*}
When $r=2$, $F_2(x,q)=1+(1-q) \sum_{n \geq 1}x^n$ tells us that
$\rchi_2(\Li^*_n(\F_q),\gl)=1-q$ for all $n \geq 1$ and all prime powers
$q$, and when $r=4$, 
 \begin{multline*}  
   F_4(x,q) =  \frac{(1-qx)^3(1-q^3x)}{(1-q^2x)^3(1-x)} \\ =
   1 + (1-q)^3(x+(3q^2+1)x^2 + (6q^4 - q^3 + 3q^2 + 1)x^3 +
   (10q^6 - 3q^5 + 6q^4 - q^3 + 3q^2 + 1)x^4 + \cdots) 
 \end{multline*}
 tells us that 
 \begin{equation*}
   \rchi_4({\Li}_n^*(\F_q),\gl) =
   \begin{cases}
     (1-q)^3 & n=1 \\
     (1-q)^3(3q^2+1) & n=2 \\
     (1-q)^3(6q^4 - q^3 + 3q^2 + 1) & n=3 \\
     (1-q)^3(10q^6 - 3q^5 + 6q^4 - q^3 + 3q^2 + 1) & n=4
   \end{cases}
\end{equation*}
for all prime powers $q$.

\begin{cor}\label{thm:Frecur}  
  $\displaystyle F_{r+1}(x,q) =
    \exp\big( - \sum_{n \geq 1} (q^n-1)^{r}
    \frac{x^n}{n} \big)$ for all $r \geq 0$.
\end{cor}

We also discuss the $p$-primary equivariant reduced \Euc s,
$\rchi_r({\Li}_n^*(\F_q),\GL n{\F_q},p)$, of the $\GL n{\F_q}$-poset
$\Li_n^*(\F_q)$
for a given prime $p$ (Definition~\ref{defn:pprimeuc}). 
The generating function for the sequence $\rchi_r({\Li}_n^*(\F_q),\GL n{\F_q},p)$, $n
\geq 1$, or
$r$th $p$-primary  generating function for short, is the power series
\begin{equation}\label{eq:primgenfct}
  F_r(x,q,p) = 1+\sum_{n \geq 1} \rchi_r({\Li}_n^*(\F_q),\gl,p)x^n \in
  \Z[[x]], \qquad r \geq 1
\end{equation}
with integer coefficients. We have $F_r(x,q,p)=1-x$ if $r=1$, or $r
\geq 1$ and $q$ is a power of $p$. When $q$ is not a power of $p$,
$F_r(x,q,p)$ depends only on the closure
$\overline{\gen q}$ of the cyclic subgroup generated by $q$ in the
topological group $\Z_p^\times$ of $p$-adic units
(Lemma~\ref{lemma:closure}).  In any case, the $(r+1)$th $p$-primary
generating function is obtained from the $(r+1)$th generating function of
Corollary~\ref{thm:Frecur} simply by replacing the factor $(q^n-1)^r$ by its
$p$-part, $(q^n-1)^r_p$.

\begin{thm}\label{thm:Frecurpprim}  
$\displaystyle
 F_{r+1}(x,q,p) =
   \exp\big( - \sum_{n \geq 1} (q^n-1)_p^{r}
    \frac{x^n}{n} \big)$
  for all $r \geq 0$.
\end{thm}

It is immediate from the elementary Lemma~\ref{lemma:varEulertrans} that
the product
expansions of the generating functions
of Corollary~\ref{thm:Frecur} and Theorem~\ref{thm:Frecurpprim}
are
\begin{alignat*}{3}
  F_{r+1}(x,q) &= \prod_{k \geq 1} (1-x^k)^{a_{r+1}(k,q)},
  &&\qquad
  &a_{r+1}(k,q) &= \frac{1}{k}\sum_{d \mid k} \mu(k/d)(q^d-1)^{r} \\
  F_{r+1}(x,q,p) &= \prod_{k \geq 1} (1-x^k)^{a_{r+1}(k,q,p)},
  &&\qquad
  &a_{r+1}(k,q,p) &= \frac{1}{k}\sum_{d \mid k} \mu(k/d)(q^d-1)_p^{r}
\end{alignat*} 
for all $r \geq 1$.




Using 
partitions we can express the equivariant reduced \Euc s more
explicitly. Let us also introduce $\rchi_{r}^{-1}({\Li}^*_n(\F_q),\GL n{\F_q})$ and
$\rchi_{r}^{-1}({\Li}^*_n(\F_q),\GL n{\F_q},p)$ for the coefficients of $x^n$ in the
reciprocal power series $F_r(x,q)^{-1}$ and $F_r(x,q,p)^{-1}$,
respectively. Then (Proposition~\ref{cor:GGL}, \eqref{eq:rchiinv}, Proposition~\ref{prop:primGGL}, \eqref{eq:FRpqx-1})
\begin{equation*}
  \rchi_{r}^{\pm 1}({\Li}^*_n(\F_q),\GL n{\F_q})
  =  \frac{1}{n!}\sum_{\lambda \vdash n} (\mp
  1)^{|\lambda|} T(\lambda) \prod_{d \in \lambda} (q^d-1)^r, \quad
  \rchi_{r}^{\pm 1}({\Li}^*_n(\F_q),\GL n{\F_q},p)
  = \frac{1}{n!}\sum_{\lambda \vdash n} (\mp
  1)^{|\lambda|} T(\lambda) \prod_{d \in \lambda} (q^d-1)_p^r
\end{equation*}
where $T(\lambda)$, for each partition $\lambda$ of $n$, is the number
of elements in the symmetric group $\Sigma_n$ of cycle type
$\lambda$. The functions $\rchi_{2}^{-1}({\Li}^*_n(\F_q),\GL n{\F_q})$
and $\rchi_{2}^{-1}({\Li}^*_n(\F_q),\GL n{\F_q},p)$ count semi-simple
and $p$-singular semi-simple classes in $\GL n{\F_q}$, respectively
(Corollary~\ref{cor:F2p}).

Tables~\ref{fig:rchigln2} and \ref{fig:rchigln23} contain examples of concrete
values of ($p$-primary)  equivariant reduced \Euc s.

Equivariant \Euc s have connections to representation theory,
combinatorics, and topology.  The Kn\"orr--Robinson conjecture
\cite{knorr_robinson:89, thevenaz92poly, thevenaz93Alperin} (a
reformulation of the (non block-wise) Alperin conjecture) predicts
that
\begin{equation*}
  \rchi_2(\cat SG{p+*},G) + z_p(G) = 0
\end{equation*}
where $\cat SG{p+*}$ is the Brown $G$-poset of nontrivial
$p$-subgroups of $G$ and $z_p(G)$ the number of irreducible complex
representations of $G$ of dimension divisible by $|G|_p$. According to
Quillen \cite[Theorem 3.1]{quillen78},
$\rchi_r(\Li^*_n(\F_q),\GL n{\F_q}) = \rchi_r(\cat S {\GL
  n{\F_q}}{s+*},\GL n{\F_q})$ where
$s$ is the characteristic of the field
$\F_q$. Since $z_s(\GL n{\F_q})=q-1$ and the second equivariant
reduced \Euc\
$\rchi_2(\Li^*_n(\F_q),\GL n{\F_q}) = 1-q$, we have verified the
Kn\"orr--Robinson conjecture for $\GL n{\F_q}$ at its defining
characteristic. This result is not new, however, as it was proved
already by Th\'evenaz \cite{thevenaz92poly}, but the approach used
here may qualify as a candidate to the combinatorial proof envisioned
in \cite[Introduction, (1)]{thevenaz92poly}. As observed also by
Th\'evenaz \cite[Theorem A, B]{thevenaz92poly}, and investigated
further in Section~\ref{sec:polyn-ident-part}, the equivariant \Euc s
of the general linear groups, and presumably also of other families of
finite groups of Lie type, lead to combinatorial \pol\ identities. The
connections to algebraic topology go through the $G$-space $|\Pi|$,
the topological realization of the $G$-poset $\Pi$.  It is convenient
now to switch to the (unreduced) equivariant \Euc s
$\chi_r(\Pi,G) = \rchi_r(\Pi,G) + |\Hom{\Z^{r-1}}{G}/G|$.  The first
equivariant \Euc\ of $(\Pi,G)$ is the usual \Euc\ of the quotient
space $|\Pi|/G$ (Proposition~\ref{prop:chi1}). The second equivariant
\Euc\ of $(\Pi,G)$ is the \Euc\ of the $G$-space $|\Pi|$ computed in
$G$-equivariant complex $K$-theory \cite[Theorem
1]{atiyah&segal89}. Finally, the $r$th $p$-primary equivariant \Euc ,
$\chi_r(\Pi,G,p)=\rchi_r(\Pi,G,p) + |\Hom{\Z_p^{r-1}}{G}/G|$, is the
\Euc\ of the homotopy orbit space $|\Pi|_{hG}$ computed in Morava
$K(r)$-theory at $p$ \cite{HKR2000} \cite[2-3, 5-1]{tamanoi2001}
\cite[Remark 7.2]{jmm:partposet2017}.

See \cite{tamanoi2001, jmm:partposet2017} for ($p$-primary)
equivariant \Euc s of Boolean and partition posets.


The following notation will be used in this article:

\begin{tabular}[h]{c|l}
  $p$ & a prime number \\
  $\nu_p(n)$ &  the $p$-adic valuation of $n$ \\
  $n_p$ &  the $p$-part of the natural number $n$ ($n_p=p^{\nu_p(n)}$) \\
  $\Z_p$ &  the ring of $p$-adic integers \\
  $q$ &  a prime power \\
  $\F_q$ &  the finite field with $q$ elements \\
  $\IM nq$ &  number of Irreducible Monic \pol s $f \in
             \F_q[t]$
             of degree $n$ with $f(0) \neq 0$\\
  $\IM n{q,p}$ &  number of Irreducible Monic \pol
                 s $f \in
             \F_q[t]$
             of degree $n$ and $p$-power order with $f(0) \neq 0$\\
  $\ord ab$ & smallest natural number $e$ such that $b^e \equiv 1
              \bmod a$ ($a$, $b$ are natural numbers with
              $\GCD{a,b}=1$) \\
$\rchi_r(n,q)$ & $\rchi_r({\Li}_n^*(\F_q),\gl)$ (Definition~\ref{defn:chir}) \\
$\rchi_r(n,q,p)$ & $\rchi_r({\Li}_n^*(\F_q),\gl,p)$ (Definition~\ref{defn:pprimeuc}) 
\end{tabular}

\section{Equivariant \Euc s of subspace posets}
\label{sec:equivariant-euc-s}


The definition \eqref{defn:rchiPiG} of the first equivariant \Euc ,
\begin{equation*}
  \rchi_1(\Pi,G) = \frac{1}{|G|} \sum_{g \in G} \rchi(C_{\Pi}(g))
\end{equation*}
closely resembles the not-Burnside lemma
\cite[Lemma~7.24.5]{stanley99}\label{lemma:burnside}
or the Lefschetz formula
\cite[Exercise 4, p 225]{dieck}.
The topological realization functor takes the $G$-poset $\Pi$ to the
$G$-space $|\Pi|$ and the following proposition is nothing surprising so the
proof will be omitted.

\begin{prop}\label{prop:chi1}
  $\rchi_1(\Pi,G) = \rchi(|\Pi|/G)$
\end{prop}


Let $f_0,f_1 \colon \Pi \to \Pi$ be two poset endo\m s of the poset
$\Pi$. We write $f_0 \leq f_1$ if $f_0(x) \leq f_1(x)$ for all
$x \in \Pi$ and $f_0 \sim f_1$ if $f_0$ and $f_1$ belong to the same
class under the equivalence relation generated by this relation.
The equivalence relation $f_0 \sim_G f_1$ between $G$-poset
endo\m s of the $G$-poset $\Pi$ is defined similarly.  The poset $\Pi$
is {\em poset contractible\/} if there exists a point $x_0 \in \Pi$ such
that $1_{\Pi} \sim x_0$ where $1_{\Pi}$ is the identity map. The
$G$-poset $\Pi$ is {\em $G$-poset-contractible\/} if there exists a point
$x_0 \in C_{\Pi}(G)$ such that $1_{\Pi} \sim_G x_0$.  The realization
of a ($G$-)poset contractible poset is a ($G$-)contractible topological space
\cite[\S1.3]{quillen78}.  If $\Pi$ is $G$-poset-contractible then the
subposets $C_{\Pi}(X)$ are poset contractible for all $X \in
\Hom{\Z^r}G$. Thus we have
\begin{equation}\label{eq:CGAcontract}
  \text{$\Pi$ is poset contractible} \implies \rchi(\Pi)=0, \qquad
  \text{$\Pi$ is $G$-poset-contractible} \implies \forall r \geq 1 \colon \rchi_r(\Pi,G)=0
\end{equation}
For instance, the Brown poset $\cat SG{p+*}$ of nontrivial
$p$-subgroups of $G$ is $G$-poset-contractible and $\rchi_r(\cat SG{p+*},G)=0$ for
all $r \geq 1$ if $G$ admits a nontrivial normal
$p$-subgroup \cite[Proposition~2.4]{quillen78}. We shall see
something similar in Lemma~\ref{lemma:preg}.

Here is a basic recursive relation between equivariant reduced
\Euc s.

\begin{lemma}\label{lemma:genrecur}
  The $r$th equivariant \Euc\ \eqref{defn:rchiPiG}  of $(\Pi,G)$ is
  \begin{equation*}
  \rchi_r(\Pi,G) = \sum_{X \in \Hom{\Z}G/G} \rchi_{r-1}(C_{\Pi}(X),C_G(X))    
  \end{equation*}
   where the sum extends over conjugacy classes of elements in $G$ and $r
   \geq 2$.
\end{lemma}
\begin{proof}
  Any homo\m\ $X \in \Hom{\Z^r}G$ corresponds to a unique pair of homo\m s $(X_1,X_2)$ with
  $X_1 \in \Hom{\Z}G$ and $X_2 \in \Hom{\Z^{r-1}}{C_G(X_1)}$.
  The subposet of $\Pi$ fixed by $X$ is the subposet fixed by $X_2$ in
  the subposet of $\Pi$ fixed by $X_1$, $C_{\Pi}(X) = C_{C_{\Pi}(X_1)}(X_2)$. The $r$th equivariant \Euc\ \eqref{defn:rchiPiG}    of $(\Pi,G)$ is
  \begin{multline*}
    \rchi_r(\Pi,G) = \frac{1}{|G|} \sum_{X \in \Hom{\Z^r}G} \rchi(C_{\Pi}(X)) =
    \frac{1}{|G|} \sum_{X_1 \in \Hom{\Z}G} \sum_{X_2 \in \Hom{\Z^{r-1}}{C_G(X_1)}} \rchi(C_{C_{\Pi}(X_1)}(X_2))  \\=
    \frac{1}{|G|} \sum_{X_1 \in \Hom{\Z}G} |C_G(X_1)| \rchi_{r-1}(C_{\Pi}(X_1),C_G(X_1)) =
    \sum_{X_1 \in \Hom{\Z}G/G} \rchi_{r-1}(C_{\Pi}(X_1),C_G(X_1))
  \end{multline*}
  where the last sum runs over conjugacy classes of elements in $G$. 
\end{proof}

We also need to know that the $r$th equivariant reduced \Euc\ is
multiplicative. For any lattice $L$, we write
$L^* = L - \{\wh 0, \wh 1\}$ for the proper part of $L$ of all
non-extreme elements.

\begin{lemma}\label{lemma:mult}   
  The function is multiplicative in the sense that
\begin{equation*}
  \rchi_r(\big(\prod_{i \in I} L_i \big)^*, \prod_{i \in I} G_i)
  =  \prod_{i \in I} \rchi_r(L_i^*,G_i)
\end{equation*}
for any finite set of $G_i$-lattices $L_i$, $i \in I$, and any $r \geq 1$.
\end{lemma}
\begin{proof}
  This follows immediately from the similar multiplicative rule,  
  $\rchi((\prod_{i \in I}L_i)^*) =  \prod_{i \in I}
  \rchi(L_i^*)$, valid for usual \Euc s. Using this property, and
  assuming for simplicity that the index set $I=\{1,2\}$ has just two
  elements, we get
\begin{gather*}
  |G_1 \times G_2|\rchi_r((L_1 \times L_2)^*, G_1 \times G_2) = 
  \sum_{X \in \Hom{\Z^r}{G_1 \times G_2}} \rchi(C_{(L_1 \times L_2)^*}(X(\Z^r)) \\ =
  \sum_{X_1 \in \Hom{\Z^r}{G_1}}\sum_{X_2 \in \Hom{\Z^r}{G_2}}
  \rchi(C_{(L_1 \times L_2)^*}((X_1\times X_2)(\Z^r))  \\
  =\sum_{X_1 \in \Hom{\Z^r}{G_1}}\sum_{X_2 \in \Hom{\Z^r}{G_2}}
  \rchi((C_{L_1}(X_1(\Z^r)) \times C_{L_2}(X_2(\Z^r)))^*) \\ =
   \sum_{X_1 \in \Hom{\Z^r}{G_1}}\sum_{X_2 \in \Hom{\Z^r}{G_2}}
   \rchi(C_{L_1^*}(X_1(\Z^r)) \times \rchi(C_{L_2^*}(X_2(\Z^r)) \\
   = \sum_{X_1 \in \Hom{\Z^r}{G_1}} \rchi(C_{L_1^*}(X_1(\Z^r))
   \sum_{X_2 \in \Hom{\Z^r}{G_2}} \rchi(C_{L_1^*}(X_1(\Z^r))  =
   |G_1| \rchi_r(L_1,G_1) |G_2| \rchi_r(L_2,G_2)
 \end{gather*}
 for any $r \geq 1$.
\end{proof}

We now turn to the case where the poset is $\Pi = {\Li}_n^*(\F_q)$ and
the group is $G=\GL n{\F_q}$.
To simplify notation, we shall often write $\rchi_r(n,q)$ for $\rchi_r({\Li}_n^*(\F_q),\gl)$.

\begin{prop}\label{prop:neq1}
  Suppose that $r=1$ or $n=1$.
  \begin{enumerate}
  \item  When $r=1$, $\rchi_1(n,q)=-\delta_{1,n}$ is $-1$ for $n=1$
    and $0$ for all $n>1$. \label{prop:neq11}
  \item   When $n=1$, $\rchi_r(1,q) =- (q-1)^{r-1}$ for all $r \geq 1$.
   \label{prop:req1}
  \end{enumerate}
\end{prop}
\begin{proof}
  The space $| {\Li}_n^*(\F_q)|$ is the simplicial complex of
  flags in $V_n(\F_q)$.  The $\gl$-orbit of a flag is described by the
  dimensions of the subspaces in the flag. Thus the quotient
  $|{\Li}_n^*(\F_q)|/\gl$ is the simplicial complex of all subsets of
  $\{1,\ldots,n-1\}$, an $(n-2)$-simplex, $\Delta^{n-2}$. By
  Proposition~\ref{prop:chi1}, $\rchi_1(n,q)$ is the usual reduced
  \Euc\ of the quotient, $\rchi(\Delta^{n-2})$, which is $-1$ when
  $n=1$ and $0$ when $n>1$. (Alternatively, this is a special case of
  Webb's theorem \cite[Proposition~8.2.(i)]{webb87}.)

  When $n=1$,
  $\Li_1^*(\F_q)=\emptyset$ is empty and  as $\rchi(\emptyset)=-1$, the
  $r$th equivariant \Euc\ is
  \begin{equation*}
  \rchi_r(1,q) = -|\Hom{\Z^r}{\GL 1{\F_q}}|/|\GL 1{\F_q}| =
  -(q-1)^{r-1}  
  \end{equation*}
for all $r \geq 1$.
\end{proof}

According to Proposition~\ref{prop:neq1}.\eqref{prop:neq11},
the first generating function,
$F_1(x,q)= 1 + \sum_{n \geq 1} \rchi_1(n,q) x^n = 1-x$, is independent of
$q$. We aim now for a recursion leading to the other generating
functions $F_r(x,q)$ for $r>1$. The next lemma reduces the problem significantly.

\begin{lemma}\label{lemma:preg}
  Let $A$ be an abelian subgroup of $\gl$ where $n > 1$. If
  $\GCD {|A|,q} \neq 1$,  then the $C_{\GL n{\F_q}}(A)$-poset
  $C_{{\Li}_n^*(\F_q)}(A)$ is $C_{\GL n{\F_q}}(A)$-contractible.
\end{lemma}
\begin{proof}
  The assumption is that the abelian group $A$ contains an element of
  order $s$, the characteristic of $\F_q$.  Let
  $F = V_n(\F_q)^{O_s(A)}$ be the subspace of vectors in $V_n(\F_q)$
  fixed by the nontrivial \syl s\ $O_s(A)$ of $A$. $F$ is a nontrivial subspace
  since $s$-groups actions on $\F_s$-vector spaces fix a nonzero
  vector \cite[Proposition VI.8.1]{brown82}. $F$ is a proper
  subspace since the nontrivial group $O_s(A)$ acts faithfully on
  $V_n(\F_q)$. $F$ is invariant under $A$ since $(vg)h = (vh)g = vg$
  for all $v \in F$, $g \in A$, $h \in O_s(A)$. Thus $F$ belongs to
  $C_{{\Li}_n^*(\F_q)}(A)$. For any $U \in C_{{\Li}_n^*(\F_q)}(A)$,
  $U \cap F = U^{O_s(A)}$ is of course proper and also
  nontrivial by \cite[Proposition
  VI.8.1]{brown82} again. Since $U \geq U \cap F \leq F$ for all
  $U \in C_{{\Li}_n^*(\F_q)}(A)$, the poset $C_{{\Li}_n^*(\F_q)}(A)$
  is poset contractible \cite[\S1.5]{quillen78}. It is even
  $C_{\GL n{\F_q}}(A)$-poset-contractible since $F^g=F$ and $U^g \cap F =
  U^g \cap F^g = (U \cap F)^g$ for all $g \in C_{\GL n{\F_q}}(A)$, $U \in C_{{\Li}_n^*(\F_q)}(A)$.
\end{proof}

\begin{cor}\label{cor:recur}
   When $r,n \geq 1$, the $(r+1)$th equivariant \Euc\ of the $\gl$-poset
  $\Li^*_n(\F_q)$ is 
  \begin{equation*}
    \rchi_{r+1}(n,q) = \sum_{\substack{[g] \in [\GL n{\F_q}] \\ \GCD {q,|g|}=1}}
    \rchi_{r}(C_{\Li_n^*(\F_q)}(g), C_{\gl}(g))
  \end{equation*}
\end{cor}
\begin{proof}
  If $G=\GL n{\F_q}$, $n \geq 2$, and
$\Pi = \Li_n^*(\F_q)$,  only the  conjugacy classes of order prime to $q$ contribute
to the sum of Lemma~\ref{lemma:genrecur} 
according to \eqref{eq:CGAcontract} and Lemma~\ref{lemma:preg}. The
corollary remains true for $n=1$ where ${\Li}_1^*(\F_q) = \emptyset$.
\end{proof}

\section{Proofs of Theorem~\ref{thm:chirgl} and Corollary~\ref{thm:Frecur}}
\label{sec:semi-simple-elements}


Let $F(x,q) = 1 + \sum_{n \geq 1} a(n,q)x^n \in 1+(x) \subseteq \Z[q][[x]]$ be a \pol\ power series with
constant term $1$ and 
$A(q)=(A_1(q),A_2(q),\ldots,A_d(q),\ldots)$ a sequence of integers
defined for every prime power $q$.
\begin{defn}\label{defn:TA}
  The $A(q)$-transform of the power series  $F(x,q)$ is the power series
  \begin{equation*}
    T_{A(q)}(F(x,q)) = \prod_{d \geq 1} F(x^d,q^d)^{A_d(q)}
  \end{equation*}
\end{defn}

 Note that $T_{A(q)} \colon 1+(x) \to 1+(x)$ is multiplicative in the sense that
 \begin{equation}\label{eq:mult}
 T_{A(q)}(1) = 1, \qquad T_{A(q)}(F_1(x,q)F_2(x,q)) = T_{A(q)}(F_1(x,q)) T_{A(q)}(F_2(x,q))
\end{equation}
for any two polynomial power series $F_1(x,q),F_2(x,q) \in 1+(x)$.

The
$\mathrm{IM}(q)$-transform will be especially important.
(See Section~\ref{sec:intro} for the definition of
$\IM dq$.)
Finite field theory tells us that
\cite[Corollary 3.21, Theorem 3.25]{lidlnieder97}
\begin{equation}\label{eq:A}
 q^n-1 = \sum_{d \mid n} d\IM dq, \qquad
  n\IM nq =  \sum_{d \mid n} \mu(n/d) (q^d-1)
\end{equation}
It is a little easier to calculate the transform with respect to the
sequence $\overline{\mathrm{IM}}(q)$ where
$\overline{\mathrm{IM}}_n(q)$ is the number of {\em all\/} Irreducible
Monic \pol s $f \in \F_q[t]$ of degree $n \geq 1$. As the two sequences agree
except in degree $1$ where $\IM 1q=q-1$ and
$\overline{\mathrm{IM}}_1(q)=q$, the two transforms,
\begin{equation}\label{eq:IMbarIM}
  F(x,q) T_{\mathrm{IM}(q)}F(x,q) = T_{\overline{\mathrm{IM}}(q)}F(x,q) 
\end{equation}
are closely related.

\begin{lemma}\label{lemma:TbarIM}
  $T_{\overline{\mathrm{IM}}(q)}((1-q^ix)^j) = (1-q^{i+1}x)^j$ and
  $T_{\mathrm{IM}(q)}((1-q^ix)^j) = \big(\frac{1-q^{i+1}x}{1-q^ix}\big)^j$ 
for any two integers $i$ and $j$.
\end{lemma}
\begin{proof}
  It is immediate that
  \begin{equation}\label{eq:classical}
    T_{\overline{\mathrm{IM}}(q)}\left( \frac{1}{1-q^ix} \right) =\prod_{d \geq 1} \frac{1}{(1-(q^ix)^d)^{\overline{\mathrm{IM}}_d(q)}}
    = \frac{1}{1-q^{i+1}x}  
 \end{equation}
because the logarithm of the middle term is $\sum_{n \geq 1}
q^n \frac{(q^ix)^n}{n}= \sum_{n \geq 1} \frac{(q^{i+1}x)^n}{n} =   -\log(1-q^{i+1}x)$
by the well-known Lemma~\ref{lemma:varEulertrans} below and \eqref{eq:A}.
The multiplicative property  \eqref{eq:mult} and \eqref{eq:IMbarIM} now imply the result
(cf.\ \cite[Chp 2]{rosen2002}).
\end{proof}

\begin{lemma}\label{lemma:varEulertrans}
  Let $(a_n)_{n \geq 1}$, $(b_n)_{n \geq 1}$, and $(c_n)_{n \geq 1}$ be integer sequences such that
  \begin{equation*}
    \prod_{n \geq 1}(1-x^n)^{-b_n} =
    \exp\big( \sum_{n \geq 1} a_n \frac{x^n}{n}\big) =
    1 + \sum_{n  \geq 1} c_n x^n
  \end{equation*}
  Then
  \begin{equation*}
    a_n = \sum_{d \mid n} db_d, \qquad
    nb_n = \sum_{d \mid n} \mu(n/d)a_d, \qquad
   nc_n = \sum_{1 \leq j \leq n} a_jc_{n-j} 
  \end{equation*}
  where $\mu$ is the number theoretic M\"obius function \cite[Chp 2,
  \S2]{irelandrosen90} and it is understood that $c_0=1$.
\end{lemma}
\begin{proof}
  The first identity follows from 
  \begin{equation*}
    \sum_{n \geq 1} a_nx^n = \sum_{n \geq 1} \sum_{k \geq 1} nb_nx^{nk} 
  \end{equation*}
  obtained by applying the operator $x\D{}x\log$
    to the given identity
  $\exp( \sum_{n \geq 1} a_n \frac{x^n}{n}) = \prod_{n \geq
    1}(1-x^n)^{-b_n}$. M\"obius inversion leads to the second
  identity. The third identity follows from
  \begin{equation*}
    \big( 1 + \sum_{n  \geq 1} c_n x^n \big) \big( \sum_{n \geq 1} a_n
    x^n \big) =
    \sum_{n \geq 1} nc_nx^n
  \end{equation*}
  obtained by applying the operator $x\frac{\mathop{d}}{\mathop{dx}}$
  to the given identity
  $\exp\big( \sum_{n \geq 1} a_n \frac{x^n}{n}\big) =
    1 + \sum_{n  \geq 1} c_n x^n$.
\end{proof}

Since an element of $\gl$ is semi-simple if and only its order is
prime to $q$ \cite[\S2]{thevenaz92poly},
it is precisely the semi-simple elements that contribute terms to the
right side in Corollary~\ref{cor:recur}.

An element $g \in \GL n{\F_q}$ is semi-simple if and only the $\F_q[t]$-module $V_n(\F_q)$ with $tv=gv$ has the form
\begin{equation*}
  V_n(\F_q) \cong \bigoplus_f \underbrace{\F_q[t]/(f(t)) \oplus \cdots \oplus \F_q[t]/(f(t))}_{m_g(f)}
\end{equation*}
where the direct sum is over irreducible monic polynomials $f \in \F_q[t]$, $f(t) \neq t$, and the $m_g(f) \geq 0$ are natural numbers. (The irreducible polynomial $f(t)=t$ of degree $1$ is excluded since we need $t$ to act as an auto\m\ on $\F_q[t]/(f(t))$.)
The Galois field $\F_q[t]/(f(t))$ has $q^{d(f)}$ elements where $d(f)$ is the degree of $f$.
Thus there is a bijective correspondence
\begin{equation*}
  \prod_f f(t)^{m(f)} \longleftrightarrow \bigoplus_f \left( \F_q[t]/(f(t)) \right)^{m(f)}
\end{equation*}
between the $q^{n-1}(q-1)$ monic \pol s in $\F_q[t]$ of degree $n$ with nonzero
constant term and the
semi-simple classes in $\GL n{\F_q}$. Here, the $f$s are  monic
irreducible \pol s with $f(0) \neq 0$, $m(f) \geq 0$ and $\sum_f d(f)m(f)=n$.

In the notation of \cite[\S1, \S2]{green55}, $g$ is semi-simple if and
only if all parts of
the associated partitions $\nu_g(f)$ are $1$ or $0$. If
$\nu(f)=(1,\ldots,1,0,\ldots)$ partitions $m(f)$,
the matrix $U_{\nu(f)}(t)$ is the zero $(m(f) \times m(f))$-matrix and
its module $V_{U_{\nu(f)}(t)}(q^{d(f)})$ is $V_{m(f)}(\F_{q^{d(f)}})$
as a $\F_{q^{d(f)}}$-vector space.
The lattice of subspaces invariant under $g$  in $V_n(\F_q)$ and the centralizer of $g$ in $\GL n{\F_q}$ are
  \begin{equation*}
   C_{\Li_n(\F_q)}(g) = \prod_f \Li_{m_g(f)}(\F_{q^{d(f)}}), \qquad 
   C_{\gl}(g) = \prod_f \GL{m_g(f)}{\F_{q^{d(f)}}}
\end{equation*}
according to \cite[Lemma 2.1]{green55}.
 The semi-simple conjugacy class $[g]$ contributes
\begin{equation*}
  \rchi_{r-1}(C_{\Li_n^*(\F_q)}(g), C_{\gl}(g)) = \prod_{f \mid c_g} \rchi_{r-1}(m_g(f),q^{d(f)}). \qquad
  c_g(t) = \prod_f f(t)^{m_g(f)}
\end{equation*}
to the sum of Corollary~\ref{cor:recur}. The product runs over the irreducible monic factors $f$ of $c_g$, the characteristic \pol\ of $g$.

It is now immediate from an extended version of the Product formula
\cite[Theorem 8.5]{bona2006} for generating functions that
Corollary~\ref{cor:recur} translates to the recurrence relation
\begin{equation}\label{eq:recurFGL}
  F_{r+1}(x,q) = T_{\mathrm{IM}(q)}F_r(x,q), \qquad r \geq 1
\end{equation}
for the generating functions \eqref{eq:genfct}. The base function is $F_1(x,q) = 1-x$ 
(Proposition~\ref{prop:neq1}.\eqref{prop:neq11}).

\begin{proof}[Proof of Theorem~\ref{thm:chirgl}]
  The sequence $F_{r+1}(x,q)$, $r \geq 0$, of Theorem~\ref{thm:chirgl}
  solves recurrence \eqref{eq:recurFGL} since
  \begin{equation*}
  T_{\mathrm{IM}(q)} \prod_{0 \leq j \leq r} (1-q^jx)^{(-1)^{r-j}\binom{r}{j}}  =
  \frac{\prod_{0 \leq j \leq r} (1-q^{j+1}x)^{(-1)^{r-j}\binom{r}{j}}}{\prod_{0 \leq j \leq r} (1-q^jx)^{(-1)^{r-j}\binom{r}{j}} }
  = \prod_{0 \leq j \leq r+1} (1-q^jx)^{(-1)^{r+1-j}\binom{r+1}{j}} 
\end{equation*}
for all $r \geq 0$
by Lemma~\ref{lemma:TbarIM}.
\end{proof}

See \cite[Proposition~4.1]{thevenaz92poly} for 
the case $r=2$ where $F_2(x,q) = \frac{1-qx}{1-x}$.  As $F_1(x,q)=1-x$
and $F_2(x,q) = \frac{F_1(x,qx)}{F_1(x,q)}$,
  \begin{equation}\label{eq:oldrecur}
    F_{r+1}(x,q) = T_{\IM{}q}F_{r}(x,q) = \frac{T_{\IM{}q}F_{r-1}(x,qx)}{T_{\IM{}q}F_{r-1}(x,q)} =
    \frac{F_{r}(x,qx)}{F_{r}(x,q)}, \qquad r \geq 2
  \end{equation}
  by induction.  This observation can be used to give another proof
  of Theorem~\ref{thm:chirgl}.

\begin{proof}[Proof of Corollary~\ref{thm:Frecur}]
The logarithm of the $(r+1)$th generating function $F_{r+1}(x,q)$  is
\begin{multline*}
  \log F_{r+1}(x,q) = \sum_{0 \leq j \leq r} (-1)^{r-j} \binom rj \log(1-q^jx) 
  = \sum_{0 \leq j \leq r} (-1)^{r-j} \binom rj \sum_{n \geq 1} -\frac{q^{nj}}{n} x^n \\
  =-\sum_{n \geq 1} \frac{x^n}{n} \sum_{0 \leq j \leq r} \binom rj (-1)^{r-j}q^{nj}
  =-\sum_{n \geq 1} (q^n-1)^r\frac{x^n}{n} 
\end{multline*}
The corollary follows.
  \end{proof}

\begin{figure}[t]
  \centering
  \begin{tabular}[t]{>{$}c<{$}|*{10}{>{$}c<{$}}}
  -\chi_r(n,2) & n=1 & n=2 & n=3 & n=4 & n=5 & n=6  & n=7 & n=8 & n=9
    & n=10\\ \hline
  r=1 &   1  &                0  &                0  &
                                                       0  &
                                                            0  &
                                                                 0 & 0
    & 0 &0 & 0\\
r= 2 &        1  &               1  &               1  &               1
                                       &               1  &
                                                            1 & 1 &1
                                                                &1 &1 \\
r= 3 &        1  &               4  &              12  &              32
                                       &              80  &
                                                            192 & 448
                                                                  &
                                                                    1024
                                                                &
                                                                  2304
                                                                & 5120 \\
 r=4 &        1  &              13  &             101  &             645
                                       &            3717  &
                                                            20101 &
                                                                    104069
                                                          & 521861 &
                                                                     2553477
    & 12252805 \\
 r=5 &        1  &              40  &             760  &           11056
                                       &          140848  &
                                                            1657216
     &18480640&198188800&2062546176&20957358080
 \end{tabular}
  \caption{Equivariant reduced \Euc s of the  $\GL
    n{\F_2}$-poset $\Li^*_n(\F_2)$}
  \label{fig:rchigln2}
\end{figure}



 We now write down explicitly the coefficient of $x^n$ in the power series
  $F_{r+1}(x,q)$ of Theorem~\ref{thm:chirgl} (Corollary~\ref{cor:rchirnq}) and 
apply Lemma~\ref{lemma:varEulertrans} to the power series of
   Corollary~\ref{thm:Frecur} (Corollary~\ref{cor:rchirecur}).

\begin{cor}\label{cor:rchirnq}  
  The $(r+1)$th, $r \geq 0$, equivariant reduced \Euc\ of the $\GL n{\F_q}$-poset ${\Li}_n^*(\F_q)$  is
  \begin{equation*}
    \rchi_{r+1}(n,q) = (-1)^n\sum_{n_0+\cdots+n_r=n} \prod_{0 \leq j \leq r}
    \binom{(-1)^{r-j}\binom{r}{j}}{n_j}q^{jn_j}
  \end{equation*}
  where the sum ranges over all $\binom{n+r}{n}$ weak compositions of
  $n$ into $r+1$ parts \cite[p 15]{stanley97}.
\end{cor}

  \begin{cor}\label{cor:rchirecur}  
    The $(r+1)$th, $r \geq 0$, equivariant reduced \Euc s satisfy the recursion
    \begin{equation*}
      \rchi_{r+1}(n,q) =
      \begin{cases}
        1 & n=0 \\
        \displaystyle -\frac{1}{n}\sum_{1 \leq j \leq n} (q^j-1)^{r} \rchi_{r+1}(n-j,q) & n>0
      \end{cases}
    \end{equation*}
  \end{cor}

For example, $\rchi_1(n,q) = -\delta_{1,n}$, $\rchi_2(n,q) = -(q-1)$,
$\rchi_3(n,q) = -n(q-1)^2q^{n-1}$ and
\begin{equation*}
  \rchi_4(n,q) = -(q-1)^3\big( 1+ \sum_{2 \leq j \leq 2n-2}
  (-1)^jd(j)q^j\big) \qquad\qquad
  d(j) =
  \begin{cases}
    \binom{(j+1)/2}{2} & 2 \nmid j \\ \binom{j/2+2}{2} & 2 \mid j
    \end{cases}
\end{equation*}
for all $n \geq 1$ (with the understanding that $\rchi_4(1,q)=-(q-1)^3$).

\subsection{Alternative presentations of the equivariant reduced \Euc s}
\label{sec:altern-pres-equiv}
  One may equally well represent the equivariant reduced \Euc s
  $\rchi_{r+1}(n,q)$
  by the generating functions
\begin{equation}\label{eq:GSp}
  \GGL{}n(x,q) =  \sum_{r \geq 0} \rchi_{r+1}(n,q) x^r  = -\delta_{1,n}
  + \rchi_2(n,q)x + \rchi_3(n,q)x^2 + \cdots, \qquad n \geq 0
\end{equation}
where the parameter $n$ is fixed rather than $r$ as in
$F_r(x,q)$ \eqref{eq:genfct}. Declaring $\rchi_{r+1}(0,q)$ to be $1$
for all $r \geq 0$, we have  $\GGL{}0(x,q) =
x +x^2 + \cdots = \frac{x}{1-x}$. 
  
  The solution to the recursion of Corollary~\ref{cor:rchirecur}
  involves integer partitions and the  following terminology.
  A multiset $\lambda$ is a base set $B(\lambda)$ with a
  multiplicity function $E(\lambda,b)$ defined for all
  $b \in B(\lambda)$. Representing the multiset as
  $\lambda=\{b^{E(\lambda,b)} \mid b \in B(\lambda) \}$, we let
\begin{alignat*}{3}
  &|\lambda| = \sum_{b \in B(\lambda)} E(\lambda,b)&&\qquad
  &&n(\lambda)=\sum_{b \in B(\lambda)} bE(\lambda,b) \\
  &T(\lambda) = \frac{n(\lambda)!}{\prod_{b \in B(\lambda)} E(\lambda,b)!
    b^{E(\lambda,b)}} &&\qquad
  &&U(\lambda,q) 
  = \prod_{b \in B(\lambda)} (q^b-1)^{E(\lambda,b)}  = \prod_{b \in \lambda} (q^b-1) 
\end{alignat*}
so that $|\lambda|$ is the cardinality (number of parts) of $\lambda$, $\lambda$
partitions $n$, $\lambda \vdash n$, if $n(\lambda)=n$, and
$T(\lambda)$ is the number of elements in the symmetric group
$\Sigma_{n(\lambda)}$ having cycle type $\lambda$ \cite[Proposition
1.1.1]{sagan:symmetric}.


\begin{prop}\label{cor:GGL}  
     For $r \geq 0$ and $n \geq 1$,
    \begin{equation*}
      \rchi_{r+1}(n,q)=
    \frac{1}{n!} 
    \sum_{ \lambda \vdash n} (-1)^{|\lambda|} T(\lambda)
    U(\lambda,q)^r \qquad
      \GGL{}n(x,q) = \frac{1}{n!}
  \sum_{\lambda \vdash n}  (-1)^{|\lambda|} \frac{ T(\lambda)}{1-x
    U(\lambda,q)}  
\end{equation*}
\end{prop}
 \begin{proof}
   The sequence $(\rchi_{r+1}(n,q))_{n \geq 1}$ as
   defined in the proposition solves
   the recursion of Corollary~\ref{cor:rchirecur}.
 \end{proof}

Examples of
Proposition~\ref{cor:GGL} are
$1!\GGL{}1(x,q) = -\frac{1}{1-x(q-1)}$, $2!\GGL{}2(x,q) =
-\frac{1}{1-x(q^2-1)} + \frac{1}{1-x(q-1)^2}$ and
\begin{equation*}
  3!\GGL{}3(x,q) = -\frac{2}{1-x(q^3-1)} + \frac{3}{1-x(q^2-1)(q-1)} - \frac{1}{1-x(q-1)^3}\\
\end{equation*}
\begin{cor}\label{cor:qm1}
  The \pol\ $\rchi_{r+1}(n,q) \in \Z[q]$,  $r \geq 0$, $n \geq 1$,  is divisible by $(q-1)^r$ 
  and also by $q^{n-1}$ when $r$ is even. The sum of the coefficients
  in the quotient \pol\ $-\rchi_{r+1}(n,q)/(q-1)^r$ is $n^{r-1}$.
  \end{cor}
 \begin{proof}
  Since $U(\lambda,q)$ is divisible by $(q-1)^{|\lambda|}$ for all 
  $\lambda \vdash n$, Proposition~\ref{cor:GGL} implies $(q-1)^r
  \mid \rchi_{r+1}(n,q)$.

  For even $r$, the weak partitions $(n_0,n_1,\ldots,n_r)$ of $n$ with
  $n_0>1$ contribute $0$ to the sum in Corollary~\ref{cor:rchirnq}
  since
  \begin{equation*}
    \binom{(-1)^r\binom{r}{0}}{n_0} = \binom{1}{n_0} =0, \qquad n_0>1
  \end{equation*}
  The remaining weak partitions with $n_0 \leq 1$ contribute $0$ or a \pol\ of
  degree
  $\sum jn_j = n_1+2n_2+\cdots+rn_r \geq n_1+n_2+\cdots+n_r \geq
  n-1$. Thus $q^{n-1} \mid \rchi_{r+1}(n,q)$.

  The quotient $U(\lambda,q)/(q-1)$ evaluates to $0$ at $q=1$ unless
  $\lambda = \{n^1\}$ where the evaluation is $n$. Thus
  $\rchi_{r+1}(n,q)/(q-1)^r$ at $q=1$ is $-\frac{1}{n!}T(\{n^1\})n^r =
  -\frac{1}{n!} \frac{n!}{n} n^r = -n^{r-1}$.
\end{proof}

\begin{rmk}  
  Let $\overline\chi_{r+1}(n,q)$ denote the quotient \pol\ $-\rchi_{r+1}(n,q)/D$ where
  $D=(q-1)^r$ for odd $r$ and $D=q^{n-1}(q-1)^r$ for even $r$. Then 
 $\overline\chi_{2}(n,q)=1$ and $\overline\chi_{3}(n,q)=n$.
The \pol\
$\overline\chi_7(2,q)= 6q^4+20q^2+6=2(3q^2+1)(q^2+3)$ is
reducible but 
I do not know of any reducible $\overline\chi_r(n,q) \in \Q[q]$
with $n>2$. 
For example,
\begin{equation*}
  \overline\chi_6(4,q) =
  35q^{12} - 45q^{11} + 150q^{10} - 170q^9 + 290q^8 - 235q^7 + 270q^6 - 100q^5 +
    60q^4 - 10q^3 + 10q^2 + 1
  \end{equation*}
  is irreducible in $\Q[q]$ by Eisenstein's irreducibility
  criterion \cite[p 78]{irelandrosen90}. The coefficient sum is $4^4$.
\end{rmk}


The reciprocal of $F_{r+1}(x,q)$,
\begin{equation*}
  F_{r+1}(x,q)^{-1} = \prod_{0 \leq j \leq r}
  (1-q^jx)^{(-1)^{r+1-j}\binom{r}{j}} =
  \exp\big( \sum_{n \geq 1} (q^n-1)^{r}\frac{x^n}{n} \big) =
  \prod_{k \geq 1} (1-x^k)^{-a_{r+1}(k,q)}
\end{equation*}
satisfies by \eqref{eq:mult} the recursion
$F_{r+1}(x,q)^{-1} = T_{\IM{}q} F_{r}(x,q)^{-1} =
T_{-\IM{}q}F_r(x,q)$.  Let $\rchi^{-1}_{r+1}(n,q) \in \Z[q]$ denote
the coefficient of $x^n$ in $F_{r+1}(x,q)^{-1}$.  In particular,
$\rchi_2^{-1}(n,q)$ is, by construction, the number of semi-simple
classes in $\GL n{\F_q}$. We have, as above,
\begin{gather}  
    1 + \sum_{n \geq 1} \rchi^{-1}_{r+1}(n,q) x^n = F_{r+1}(x,q)^{-1}   \qquad
    1 + \sum_{r \geq 1} \rchi^{-1}_{r+1}(n,q) x^r = \frac{1}{n!}
    \sum_{\lambda \vdash n} \frac{T(\lambda)}{1-xU(\lambda,q)} \\
    \rchi^{-1}_{r+1}(n,q)   = \frac{1}{n!}\sum_{ \lambda \vdash n}  T(\lambda)
    U(\lambda,q)^r 
    = (-1)^n\sum_{n_0+\cdots+n_r=n} \prod_{0 \leq j \leq r}
    \binom{(-1)^{r+1-j}\binom{r}{j}}{n_j}q^{jn_j} \label{eq:rchiinv}
  \end{gather}  
 Special cases are $\rchi^{-1}_1(n,q) =1$,
  $\rchi^{-1}_2(n,q) = q^n-q^{n-1} $,
  $\rchi^{-1}_3(n,q) = \frac{q-1}{q+1}(q^{2n}-1)$ and
\begin{equation*}
  \rchi^{-1}_4(n,q) = (q-1)^3q^{n-1}\big(q^{2n-2}+\sum_{2\leq j \leq 2n-2}
  (-1)^jd(j)q^{2n-2-j}\big) \qquad
   d(j) =
  \begin{cases}
    \binom{(j+1)/2}{2} & 2 \nmid j \\ \binom{j/2+2}{2} & 2 \mid j
  \end{cases}
\end{equation*}
with the understanding that $\rchi^{-1}_4(1,q)=(q-1)^3$. For
$n=1,2,3$,  $r>0$, 
$1!\rchi^{\pm 1}_{r+1}(1,q) = \mp (q-1)^r$, $2!\rchi^{\pm 1}_{r+1}(2,q) =
\mp (q^2-1)^r + (q-1)^{2r}$ and $3!\rchi^{\pm 1}_{r+1}(3,q) = \mp
2(q^3-1)^r + 3(q^2-1)^r(q-1)^r \mp (q-1)^{3r}$ by
Proposition~\ref{cor:GGL} and \eqref{eq:rchiinv}.

\subsection{Polynomial identities for partitions}
\label{sec:polyn-ident-part}
The  \pol\ identities \cite[Theorem A,
B]{thevenaz92poly} are parts of a greater hierarchy of \pol\ identities.

Let $M_n$, $n \geq 1$, be the set of all finite multisets
$\lambda = \{(m_1,d_1)^{e(m_1,d_1)}, \cdots, (m_s,d_s)^{e(m_s,d_s)}\}$
of pairs of natural numbers $(m_i,d_i)$ with multiplicities
$e(m_i,d_i)$ such that the multiset
$\{(m_1d_1)^{e(m_1,d_1)}, \cdots, (m_sd_s)^{e(m_s,d_s)}\}$ is a
partition of $n$. The coefficient of $x^n$ in the $A(q)$-transform
(Definition~\ref{defn:TA}) $T_{A(q)}F(x,q)$ of
$F(x,q)= 1 + \sum_{n \geq 1} a(n,q)x^n$ is
\begin{equation*}
   \sum_{\lambda \in M_n}
   \prod_{\{ d \mid \exists m \colon
     (m,d) \in B(\lambda) \}}
  \binom{A_d(q)}{[E(\lambda,(m,d)) \mid  (m,d) \in B(\lambda)]}
   \prod_{\{m \mid  (m,d) \in B(\lambda)\}}a(m,q^d)^{E(\lambda,e(m,d))}
\end{equation*}
where
\begin{itemize}
\item the first  product extends over the set of all 
  second coordinates of the  multiset $\lambda$
\item $[E(\lambda,(m,d)) \mid (m,d) \in B(\lambda)]$
  is the multiset of multiplicities of elements of $\lambda$ with
  $d$ as second coordinate
\item the multinomial coefficient
  \begin{equation*}
  \binom{n}{k_1,\ldots,k_s} =\frac{n(n-1) \cdots (n+1-\sum
      k_i)}{k_1!k_2! \cdots k_s!}    
  \end{equation*}
 \end{itemize}
For instance, the multiset $\{(1,1)^2,(2,1)^2,(1,2)^2\}$ from
$M_{10}$ contributes the term
\begin{equation*}
  \binom{A_1(q)}{2,2} a(1,q)^2 a(2,q)^2 \binom{A_2(q)}{2}a(2,q^2)^2 
\end{equation*}
to the sum over all the $244$ multisets in $M_{10}$.

The ordinary generating function for the number of elements in $M_n$ is
\begin{equation*}
  1+\sum_{n \geq 1} |M_n|x^n =
  \prod_{k \geq 1} (1-x^k)^{-\tau(k)}
\end{equation*}
where $\tau(k)$ is the number of divisors of $k$. The first terms
are $|M_n|=1, 3, 5, 11, 17, 34, 52, 94, 145,244,\ldots$.

Proposition~\ref{cor:GGL}, \eqref{eq:rchiinv} and the recursive relations
$F_{r+1}(x,q)^{\pm 1} = T_{\IM{}q}
F_r(x,q)^{\pm 1} = T_{-\IM{}q}
F_r(x,q)^{\mp 1}$ give a sequence of \pol\ identities
\begin{equation*}
  \rchi_{r+1}^{-\varepsilon_1}(n,q) =
  \sum_{\lambda \in M_n}
   \prod_{\{ d \mid \exists m \colon
     (m,d) \in B(\lambda) \}}
  \binom{ \varepsilon_1 \varepsilon_2 \IM dq}{[E(\lambda,(m,d)) \mid  (m,d) \in B(\lambda)]}
   \prod_{\{m \mid (m,d) \in B(\lambda)\}}
   \rchi_r^{-\varepsilon_2}(m,q^d), \qquad r \geq 1
\end{equation*}
where $\rchi_{r+1}^{-\varepsilon_1}(n,q) =
  \frac{1}{n!} \sum_{\lambda \vdash n}
  \varepsilon_1^{|\lambda|}T(\lambda)U(\lambda,q)^{r}$,
 $\rchi_r^{-\varepsilon_2}(m,q^d) = \sum_{\mu \vdash
   m}\varepsilon_2^{|\mu|}T(\mu)U(\mu,q^d)^{r-1}$ for $\varepsilon_1, \varepsilon_2 = \pm 1$.
Taking $r=1$ and $r=2$ we get
 the \pol\ identities 
 \begin{align*}
   &\rchi_2^{-\varepsilon_1}(n,q) =
     \frac{1}{n!}
    \sum_{\lambda \vdash n}
    \varepsilon_1^{|\lambda|}T(\lambda)U(\lambda,q) =
     \begin{cases}
      \displaystyle \sum_{\lambda \vdash n} (-1)^{|\lambda|} \prod_{d \in
    B(\lambda)} \binom{-\varepsilon_1 \IM dq}{E(\lambda,d)} &
  \varepsilon_2=-1\\
  \displaystyle \sum_{\lambda \in M_n} \prod_d \binom{\varepsilon_1\IM
                                                                 dq}{[E(\lambda,(m,d))
                                                                 \mid
                                                                 (m,d)\in
                                                                 B(\lambda)]}  &
  \varepsilon_2=+1
                                                             \end{cases}
   \\
   &\rchi_3^{-\varepsilon_1}(n,q) =
\frac{1}{n!}
    \sum_{\lambda \vdash n}
    \varepsilon_1^{|\lambda|}T(\lambda)U(\lambda,q)^2 =
     \begin{cases}
       \displaystyle \sum_{\lambda \in M_n }
   \prod_{d} \binom{-\varepsilon_1\IM dq}{[E(\lambda,(m,d)) \mid (m,d) \in B(\lambda)]}
                                  \prod_{m} (1-q^d)^{E(\lambda,(m,d))}
                                  & \varepsilon_2=-1\\
        \displaystyle \sum_{\lambda \in M_n} \prod_d \binom{\varepsilon_1\IM
     dq}{[E(\lambda,(m,d)) \mid (m,d) \in B(\lambda)]} \prod_m
    (q^{dm}-q^{d(m-1)})^{E(\lambda,(d,m))}  & \varepsilon_2=+1                         
     \end{cases}
 \end{align*}
where we used that $\rchi_1(m,q^d)=-\delta_{1,m}$
(Proposition~\ref{prop:neq1}.\eqref{prop:neq11}) contributes only for
$m=1$,
$\rchi^{-1}_1(m,q^d)=1$,
$\rchi_2(m,q^d)=1-q^d$ and
$\rchi^{-1}_2(m,q^d)=q^{dm}-q^{d(m-1)}$.
The left sides above are
\begin{equation*}
  \rchi_2^{-\varepsilon_1}(n,q)=
  \begin{cases}
    q^n-q^{n-1} & \varepsilon_1=+1 \\ 1-q & \varepsilon_1=-1  
  \end{cases} \qquad
  \rchi_3^{-\varepsilon_1}(n,q) =
  \begin{cases}
    \frac{q-1}{q+1}(q^{2n}-1) & \varepsilon_1=+1 \\ -nq^{n-1}(q-1)^2 & \varepsilon_1=-1  
  \end{cases}
\end{equation*}
The \pol\ identities \cite[Theorem A, B]{thevenaz92poly} are the
identities at $r=1$ for $\rchi_2^{\pm}(n,q)$. The \pol\ identities for
$r>1$ and the identities involving 
Corollary~\ref{cor:GGL} seem to be new.

Specializing further to $n=3$, the
index set
\begin{equation*}
M_3=\{ \{(3,1)\}, \{(1,3)\}, \{(1,1),(1,2)\}, \{(1,1),(2,1)\},
\{(1,1)^3\} \}
\end{equation*}
contains $5$ multisets and the above identities for $\rchi_2^{\pm \varepsilon_1}(3,q)$ are
\begin{align*}  
      \rchi_2^{-1}(3,q) &= q^3-q^2 = \frac{1}{6}(2(q^3-1)+3(q^2-1)(q-1)+(q-1)^3) 
                                      \\
              &\mathrel{\substack{\varepsilon_2=-1 \\ = \\ \text{A}}}
  -\binom{-\IM 3q}{1} + \binom{-\IM 2q}{1}\binom{-\IM 1q}{1} -
          \binom{-\IM 1q}{3}
  \\ &\stackrel{\varepsilon_2=+1}{=}
      \binom{\IM 3q}{1} + \binom{\IM 1q}{1} + \binom{\IM 1q}{1}
      \binom{\IM 2q}{1} +
       \binom{\IM 1q}{1,1} +\binom{\IM 1q}{3}  \\
   \rchi_2(3,q) &= 1-q = \frac{1}{6}(-2(q^3-1)+3(q^2-1)(q-1)-(q-1)^3) 
  \\ &\mathrel{\substack{\varepsilon_2=-1 \\ = \\ \text{B}}}
  -\binom{\IM 3q}{1} + \binom{\IM 2q}{1}\binom{\IM 1q}{1} -
             \binom{\IM 1q}{3} 
  \\ &\stackrel{\varepsilon_2=+1}{=}
      \binom{-\IM 3q}{1} + \binom{-\IM 1q}{1} + \binom{-\IM 1q}{1}
      \binom{-\IM 2q}{1} +
                \binom{-\IM 1q}{1,1} +\binom{-\IM 1q}{3}   
\end{align*}
while for $\rchi_3^{\pm \varepsilon_1}(3,q)$ they are
\begin{align*}
  \rchi_3^{-1}(3,q) &=
  \frac{q-1}{q+1}(q^{6}-1) =
                             \frac{1}{6}(2(q^3-1)^2+3(q^2-1)^2(q-1)^2+(q-1)^6)
 \\&\stackrel{\varepsilon_2=-1}{=}
      \binom{-\IM 3q}{1}(1-q^3) + \binom{-\IM 1q}{1}(1-q) + \binom{-\IM 1q}{1}
      \binom{-\IM 2q}{1}(1-q)(1-q^2) \\ &+
                                          \binom{-\IM 1q}{1,1}(1-q)^2
                                          +\binom{-\IM 1q}{3}(1-q)^3
      \\&\stackrel{\varepsilon_2=+1}{=}
      \binom{\IM 3q}{1}(q^3-1) + \binom{\IM 1q}{1}(q^3-q^2) + \binom{\IM 1q}{1}
      \binom{\IM 2q}{1}(q-1)(q^2-1) \\ &+
                \binom{\IM 1q}{1,1}(q-1)(q^2-q) +\binom{\IM
                                         1q}{3}(q-1)^3\\
      \rchi_3(3,q) &=
     -3(q-1)^2q^2 = \frac{1}{6}(-2(q^3-1)^2+3(q^2-1)^2(q-1)^2-(q-1)^6) 
 \\&\stackrel{\varepsilon_2=-1}{=}
      \binom{\IM 1q}{1}(1-q) + \binom{\IM 3q}{1}(1-q^3) + \binom{\IM 1q}{1}
      \binom{\IM 2q}{1}(1-q)(1-q^2) \\ &+
                                         \binom{\IM 1q}{1,1} (1-q)^2+\binom{\IM 1q}{3}(1-q)^3
                                         \\&\stackrel{\varepsilon_2=+1}{=}
      \binom{-\IM 3q}{1}(q^3-1) + \binom{-\IM 1q}{1}(q^3-q^2) + \binom{-\IM 1q}{1}
      \binom{-\IM 2q}{1}(q-1)(q^2-1) \\ &+
                \binom{-\IM 1q}{1,1}(q-1)(q^2-q) +\binom{-\IM
                                          1q}{3}(q-1)^3
    \end{align*}

 \section{The $p$-primary equivariant reduced \Euc}
\label{sec:p-prim-equiv}


The \emph{$r$th $p$-primary equivariant reduced \Euc\/} of the
$G$-poset $\Pi$ 
is  the normalized sum \cite[(1-5)]{tamanoi2001}
\begin{equation}\label{defn:pprimrchiPiG}
  \rchi_r(\Pi,G,p) = \frac{1}{|G|} \sum_{X \in \Hom{\Z \times \Z_p^{r-1}}G}
  \rchi(C_{\Pi}(X(\Z \times \Z_p^{r-1}))
\end{equation}
of the reduced \Euc s of the $X(\Z \times \Z_p^{r-1})$-fixed
$\Pi$-subposets as $X$ ranges over the set of all homo\m s of
$\Z \times \Z_p^{r-1}$ to $G$. When $G$ acts trivially on $\Pi$,
$\rchi_r(\Pi,G,p) = \rchi(\Pi) |\Hom{\Z_p^{r-1}}{G}/G|$ is
proportional to the number of conjugacy classes of commuting
$(r-1)$-tuples of $p$-singular elements of $G$ \cite[Lemma
4.13]{HKR2000}.  (A group element is {\em $p$-singular\/} if its order
is a power of $p$ \cite[Definition 40.2, \S 82.1]{cr}.)  When $p$ does
not divide the order of $G$, there are no nontrivial $p$-singular
elements in $G$ and $\rchi_r(\Pi,G,p) = \rchi_1(\Pi,G)$ does not
depend on $r$. In particular, the
primary equivariant \Euc s 
of the $\GL n{\F_q}$-poset $\Li^*_n(\F_q)$ are
defined as follows.

\begin{defn}\label{defn:pprimeuc}
  The $r$th, $r \geq 1$, $p$-primary equivariant reduced \Euc\ of the $\gl$-poset
  $\Li_n^*(\F_q)$ is the normalized sum
\begin{equation*}
  \rchi_r(\Li_n^*(\F_q), \GL n{\F_q},p) =
  \frac{1}{|\gl|}
  \sum_{X \in \Hom{\Z \times
      \Z_p^{r-1}}{\gl}}\rchi(C_{\Li_n^*(q)}(X(\Z \times \Z_p^{r-1})))
\end{equation*}
of reduced \Euc s.
\end{defn}

In this section we calculate the $p$-primary generating functions
$F_r(x,q,p) = 1 + \sum_{n \geq 1}\rchi_r(n,q,p)x^n$
\eqref{eq:primgenfct} for the $p$-primary equivariant reduced \Euc s,
$\rchi_r(n,q,p) =\rchi_r(\Li_n^*(\F_q), \GL n{\F_q},p)$.

\begin{prop}\label{prop:pprimr1n1}
  Suppose that $r=1$ or $n=1$.
  \begin{enumerate}
  \item When $r=1$, $\rchi_1(n,q,p) = \rchi_1(n,q) = -\delta_{1,n}$ is $-1$ for
    $n=1$ and $0$ for $n>1$. \label{prop:pprimr1n11}
  \item When $n=1$, $\rchi_r(1,q,p) = -(q-1)^{r-1}_p$ for all $p$,
    $q$, and $r \geq 1$. \label{prop:pprimr1n1item2}
  \end{enumerate}
\end{prop}
\begin{proof}
   When $r=1$, the $p$-primary equivariant reduced \Euc\ and the
   equivariant reduced \Euc\ agree by Definition~\ref{defn:pprimeuc}
   and we refer to Proposition~\ref{prop:neq1}.\eqref{prop:neq11}.
  When $n=1$, 
  \begin{equation*}
    \rchi_r(1,q,p) =-|\Hom{\Z \times \Z^{r-1}_p}{\GL 1{\F_q}}|/|\GL 1{\F_q}| =
    -(q-1)(q-1)_p^{r-1}/(q-1) = -(q-1)_p^{r-1} 
  \end{equation*}
  since $\Li^*_1(\F_q)=\emptyset$ and $\rchi(\emptyset)=-1$.
\end{proof}

According to Proposition~\ref{prop:pprimr1n1}.\eqref{prop:pprimr1n11},
the first $p$-primary generating function $F_1(x,q,p)=1-x$ is
independent of $p$ and $q$. In fact, $F_r(x,q,p)=1-x$ for all
$r \geq 1$ if $q$ is a power of $p$ by Lemma~\ref{lemma:preg}. The
interesting case is thus when $p \nmid q$ where the first terms in the
$r$th generating function are
$F_r(x,q,p) = 1-(q-1)_p^{r-1}x + \cdots$.

The analogue of Corollary~\ref{cor:recur}, proved exactly as before, asserts that
\begin{equation*}
  \rchi_r(n,q,p) = \sum_{[g] \in [\gl_p]} \rchi_{r-1}(C_{\Li_n^*(\F_q)}(g),C_{\gl}(g),p)
\end{equation*}
where the sum is extended over the set $[\gl_p]$ of $p$-singular
conjugacy classes in $\GL n{\F_q}$.

The {\em order\/} of a \pol\ $f \in \F_q[t]$ with $f(0) \neq 0$ is the
least positive integer $e$ for which $f(t)$ divides $t^e-1$ \cite[Definition 3.2]{lidlnieder97}.


\begin{lemma}\label{lemma:psemisimple}
  A semi-simple element of $\GL n{\F_q}$ is $p$-singular if and
  only if all irreducible factors of its characteristic \pol\ have
  $p$-power order.
\end{lemma}
\begin{proof}
  Is is enough to show that multiplication by $t$ on $\F_q[t]/(f(t))$,
  where $f(t)$ is an irreducible monic \pol\ with $f(0) \neq 0$,
  has $p$-power order if and only if $f$ has $p$-power order. But
  multiplication by $t$ has $p$-power order 
  if and only if $f(t)$ divides $t^{c}-1$ for some $p$-power $c$ if and only if 
  the order of $f$ divides $c$ by \cite[Lemma 3.6]{lidlnieder97}.
\end{proof}

As in
Section~\ref{sec:semi-simple-elements} we conclude from Lemma~\ref{lemma:psemisimple} that the
$p$-primary generating functions obey the recurrence relation
\begin{equation}
  \label{eq:recurFprim}
  F_{r+1}(x,q,p) = T_{\IM{}{q,p}} F_r(x,q,p), \qquad r \geq 1
\end{equation}
with base function $F_1(x,q,p)=1-x$.


We need a little preparation before we can solve \eqref{eq:recurFprim}.
The following observation is  the $p$-primary analogue of a
fundamental classical result. 

\begin{thm}\label{thm:thm320prim}
  The product of all monic irreducible \pol s in $\F_q[t]$ with nonzero
  constant term, $p$-power order, and degree dividing $n \geq 1$ is $t^{(q^n-1)_p}-1$.
\end{thm}
\begin{proof}
  We already know from the classical theorem \cite[Theorem 3.20]{lidlnieder97}
  that each irreducible
  factor in $t^{q^n-1}-1$ occurs exactly once in the factorization. If
  $f$ is an irreducible factor of $t^{(q^n-1)_p}-1$, then the order of
  $f$ divides $(q^n-1)_p$ by \cite[Corollary
  3.7]{lidlnieder97}. Conversely, let $f$, $f(0) \neq 0$, be a monic
  irreducible \pol\ of degree dividing $n$ and of order $p^e$ for some
  $e \geq 0$. Then $f$ divides $t^c-1$ where $c=\GCD{p^e,q^n-1}$
  and hence $f$ also divides
  $t^{(q^n-1)_p}-1$ \cite[Lemma 3.6, Corollary 3.7]{lidlnieder97}.
\end{proof}

By comparing the degree of $t^{(q^n-1)_p}-1$ with the total degree of
its canonical factorization \cite[Theorem 1.59]{lidlnieder97} we
obtain $p$-primary versions
\begin{equation}
  \label{eq:IMdpq}
  (q^n-1)_p = \sum_{d \mid n} d\IM d{q,p}, \qquad
  n\IM n{q,p} =  \sum_{d \mid n} \mu(n/d) (q^d-1)_p
\end{equation}
of the classical relations \eqref{eq:A}. See Section~\ref{sec:intro}
for the definition of $\IM d{p,q}$.

We are now ready to prove Theorem~\ref{thm:Frecurpprim}. The present
proof, a tremendous improvement of the original lengthy case-by-case
checking, is due to an anonymous referee. A similar argument can be
used to prove Theorem~\ref{thm:chirgl}. 

\begin{proof}[Proof of Theorem~\ref{thm:Frecurpprim}]
  We must show  that the power series
  \begin{equation*}
  F_{r}(x,q,p) = \prod_{n \geq 1} (1-x^n)^{a_{r}(n,q,p)},
  \qquad a_{r}(n,q,p) = \frac{1}{n}\sum_{d \mid n} \mu(n/d)(q^d-1)_p^{r-1}  
  \end{equation*}
solve recurrence \eqref{eq:recurFprim}. Indeed, the $\IM{}{q,p}$-transform of
$F_r(x,q,p)$ equals $F_{r+1}(x,q,p)$ because in the product
\begin{equation*}
  T_{\IM{}{q,p}}F_r(x,q,p) = \prod_{d \geq 1} F_r(x^d,q^d,p)^{\IM d{q,p}}
  = \prod_{d,n \geq 1} (1-x^{nd})^{\IM d{q,p}a_r(n,q^d,p)} =
  \prod_{N \geq 1} (1-x^N)^{\sum_{d \mid N} a_r(N/d,q^d,p) \IM
    d{q,p}} 
\end{equation*}
the exponent of the $(1-x^N)$-factor is
\begin{multline*}
  \sum_{d\mid N} a_r(N/d,q^d,p) \IM{d}{q,p} =
  \sum_{d \mid N} \frac{d}{N} \sum_{e \mid (N/d)}
  \mu(N/de)(q^{de}-1)^{r-1}_p \IM{d}{q,p} \\ \stackrel{\text{\eqref{eq:IMdpq}}}{=}
  \frac{1}{N} \sum_{d \mid N} \sum_{e \mid (N/d)}
  \mu(N/de)(q^{de}-1)^{r-1}_p \sum_{f \mid d} \mu(d/f)(q^f-1)_p =
  \frac{1}{N} \sum_{f \mid d_1 \mid d_2 \mid N}
  \mu(N/d_2)(q^{d_2}-1)^{r-1}_p \mu(d_1/f)(q^f-1)_p \\=
  \frac{1}{N} \sum_{d \mid N} \mu(N/d)(q^d-1)^{r-1}_p(q^d-1)_p =
  \frac{1}{N} \sum_{d \mid N} \mu(N/d)(q^d-1)^{r}_p =
  a_{r+1}(N,q,p)
\end{multline*}
In this calculation we used that, for fixed $f$
and $d_2$, the sum
\begin{equation*}
  \sum_{d_1 \colon f \mid d_1 \mid d_2} \mu(d_1/f) =
  \begin{cases}
    1 & f=d_2 \\ 0 & f < d_2 
  \end{cases}
\end{equation*}
contributes only when $f=d_1=d_2$.
\end{proof}

\begin{cor}\label{cor:rpchirecur}     
  The $(r+1)$th, $r \geq 0$, $p$-primary equivariant reduced \Euc s satisfy the recursion
    \begin{equation*}
      \rchi_{r+1}(n,q,p) =
      \begin{cases}
        1 & n=0 \\
        \displaystyle -\frac{1}{n}\sum_{1 \leq j \leq n} (q^j-1)_p^{r} \rchi_{r+1}(n-j,q,p) & n>0
      \end{cases}
    \end{equation*}
  \end{cor}
  \begin{proof}
   Apply Lemma~\ref{lemma:varEulertrans} to the formula of
   Theorem~\ref{thm:Frecurpprim}. 
  \end{proof}

\begin{figure}[t]
  \centering
  \begin{tabular}[t]{>{$}c<{$}|*{10}{>{$}c<{$}}}
  -\chi_r(n,2,3) & n=1 & n=2 & n=3 & n=4 & n=5 & n=6  & n=7 & n=8 & n=9
    & n=10\\ \hline
r=1 & 1 & 0 & 0 & 0 & 0 & 0 & 0 & 0 & 0 & 0 \\
r=2 & 1 & 1 & -1 & 0 & 0 & 1 & -1 & -1 & 1 & 0 \\
r=3 & 1 & 4 & -4 & -6 & 6 & 16 & -16 & -49 & 49 & 72 \\
r=4 & 1 & 13 & -13 & -78 & 78 & 403 & -403 & -2236 & 2236 & 10413 \\
r=5 & 1 & 40 & -40 & -780 & 780 & 10960 & -10960 & -134590 & 134590 &
                                                                 1500408 
 \end{tabular}
  \caption{$3$-primary equivariant reduced \Euc s of the  $\GL
    n{\F_2}$-poset $\Li^*_n(\F_2)$}
  \label{fig:rchigln23}
\end{figure}


 We now look more closely at the sequence
 $\IM{}{q,p} = (\IM{d}{q,p})_{d \geq 1}$ recording  the number of
 irreducible monic \pol s of $p$-power order, nonzero constant term,
 and degree $d$ in $\F_q[t]$.  If $q$ is a power of $p$,
 $\IM 1{q,p}=1$ and $\IM n{q,p}=0$ for all $n > 1$, as the only \pol\
 that fulfills the requirements is $f(t)=t-1$ \cite[Corollary
 3.2]{lidlnieder97}.
In the
more interesting case where $p$ and  $q$ are prime, consider the subgroup
$\gen q$ of $\Z_p^\times$ generated by $q$ in the unit topological
group $\Z_p^\times$ of the ring $\Z_p$ of $p$-adic integers.

  \begin{lemma}\label{lemma:closure}
    When $p \nmid q$, the sequence $\IM{}{q,p}$ and the function
    $F_r(x,q,p)$, $r \geq 1$, depend only on the closure
    $\overline{\gen q}$ in $\Z_p^\times$ of $\gen q$.
  \end{lemma}
  \begin{proof}
    The integer $\IM d{q,p}$ depends only on the images
    of $\gen q$ under the continuous \cite[Chp 1, \S3]{serre73} homo\m
    s $\Z_p^\times \to (\Z/p^n\Z)^\times$, $n \geq 1$.  But $\gen q$ and
    $\overline{\gen q}$ have the same image in the discrete
    topological space $(\Z/p^n \Z)^\times$.
  \end{proof}

  We say that $q_1$ and $q_2$, prime powers prime to $p$, are
  $p$-equivalent if $\overline{\gen {q_1}}=\overline{\gen {q_2}}$ in
  $\Z_p^\times$.  More explicitly, 
$q_1$ and $q_2$ are $p$-equivalent if and only if
  $O(q_1,p) = O(q_2,p)$ \cite[\S3]{bmo1} where,
  for a prime power $q$ prime to $p$, 
  $O(q,p)$ denotes the integer pair 
  \begin{equation}\label{defn:nupq}
    O(q,p) =
    \begin{cases}
      (q \bmod 8, \nu_2(q^2-1)) & p=2 \\
      (\ord pq, \nu_p(q^{\ord pq}-1)) & p>2 
    \end{cases}
  \end{equation}
  The multiplicative order, $\ord pq$, was defined in
  Section~\ref{sec:intro}.  The following well-known lemma can be used
  to calculate $p$-adic valuations.

\begin{lemma}[Lifting the Exponent]\label{lemma:LTE}
  Let $p$ be any prime and $n \geq 1$ any natural number.
  \begin{enumerate}

  \item If $a \equiv b \not\equiv 0 \bmod p$ and $\gcd(p,n) =1$ then
      $\nu_p(a^n-b^n) = \nu_p(a-b)$

  \item 
    If $p$ is odd and $a \equiv b \not\equiv 0 \bmod p$ then
    $\nu_p(a^n-b^n) = \nu_p(a-b) + \nu_p(n)$

  \item If $a$ and $b$ are odd and $n$ even then
            $\nu_2(a^n-b^n) = \nu_2(a-b)+\nu_2(a+b)+\nu_2(n)-1$.

\item          
    If $a$ and $b$ are odd and $a \equiv b \bmod 4$ then
      $\nu_2(a^n-b^n) = \nu_2(a-b) + \nu_2(n)$
      
  \end{enumerate}
\end{lemma}

 We first consider the situation when $p$ is an {\em odd\/} prime.  Let
$g$ be a prime primitive root mod $p^2$ \cite[Definition p
41]{irelandrosen90}. Such a prime $g$ always exists by the Dirichlet
Density Theorem \cite[Chp 16, \S1, Theorem 1]{irelandrosen90} and the
congruence class of $g$ generates $(\Z/p^n\Z)^\times$ for all
$n \geq 1$ \cite[Chp 4, \S1, Theorem 2]{irelandrosen90}.  By
\cite[Lemma 1.11.(a)]{bmo2} it suffices to consider $p$-primary
generating functions $F_r(x, (g^s)^{p^e},p)$ at the prime powers
$(g^s)^{p^e}$ where $s$ divides $p-1$ and $e \geq 0$.

  \begin{lemma}\label{lemma:gpm1}
    Let $p$ be an odd prime and  $q=g^{p-1}$. For all $n \geq 1$ and $r \geq 0$,
    \begin{equation*}
      \qquad F_{r+1}(x,q,p) = \exp( -\sum\limits_{n \geq 1} (pn)_p^r
      \frac{x^n}{n}), \qquad
      \IM n{q,p} =
      \begin{cases}
        p & n =1 \\
        p-1 & n = n_p > 1 \\
        0 & \text{otherwise}
      \end{cases}
    \end{equation*}
    and, for any $e \geq 0$, $\IM n{q^{p^e},p} = p^e \IM n{q,p}$ and
      $F_{r+1}(x,q^{p^e},p) = F_{r+1}(x,q,p)^{p^{re}}$.
  \end{lemma}
  \begin{proof}
    Since $(q^n-1)_p = pn_p$, Theorem~\ref{thm:Frecurpprim} immediately gives
    the formula for $F_{r+1}(x,q,p)$.  An elementary calculation verifies
    $\sum_{d \mid n} d \IM d{q,p} = (q^n-1)_p$ when the integers $\IM
    d{q,p}$ are defined as in the lemma. Since $(q^{dp^e}-1)_p =
    p^e(q^d-1)_p$, we get $\IM n{q,p^{p^e}} = p^e \IM n{q,p}$ by
    \eqref{eq:IMdpq} and 
    $F_{r+1}(x,q^{p^e},p) = F_{r+1}(x,p,q)^{p^e}$ by Theorem~\ref{thm:Frecurpprim}.
  \end{proof}

  \begin{lemma}\label{lemma:Fpexpansion}
    $\displaystyle \exp(-\sum_{n \geq 1} (pn)_p^r \frac{x^n}{n}) =  \prod_{n \geq 0} \Big(
    \frac{(1-x^{p^n})^p}{1-(x^{p^n})^p} \Big)^{p^{(r-1)(n+1)}}$ for
      any prime $p$.
  \end{lemma}
  \begin{proof}
    Let $F(x) = \exp(-\sum_{n \geq 1} (pn)_p^r \frac{x^n}{n})$.
     The rewriting
    \begin{equation*}
      -\sum_{n \geq 1} (pn)_p^r\frac{x^n}{n} =
      -\sum_{p \nmid n} p^r\frac{x^n}{n} -\sum_{p \mid n}
      (pn)_p^r\frac{x^n}{n} =
      -p^{r-1}\sum_{p \nmid n} p\frac{x^n}{n}
       - p^{r-1}\sum_{n \geq 1} (pn)_p^r\frac{(x^p)^n}{n} 
    \end{equation*}
    translates to the functional equation 
    \begin{equation*}
      F(x) = \Big( \frac{(1-x)^p}{1-x^p} \Big)^{p^{r-1}} F(x^p)^{p^{r-1}}
    \end{equation*}
    Repeated use of this relation
    leads to the product expansion of the lemma.
  \end{proof}

  \begin{cor}
    Let $p$ be an odd prime.
    When $(q-1)_p= p$, i.e.\ $O(q,p)=(1,1)$, and $r \geq 0$
    \begin{equation*}
      F_{r+1}(x,q,p) = 
      \prod_{n \geq 0} \Big(
      \frac{(1-x^{p^n})^p}{1-(x^{p^n})^p} \Big)^{p^{(r-1)(n+1)}}
    \end{equation*}
  \end{cor}
  \begin{proof}
    Combine Lemma~\ref{lemma:gpm1} and Lemma~\ref{lemma:Fpexpansion}.
  \end{proof}

  \begin{lemma}\label{lemma:gs}
    Let $p$ be an odd prime and $q=g^s$ where $s \geq 1$ and $st=p-1$ for some $t>1$. Then 
    \begin{equation*}
        F_{r+1}(x,q^{p^e},p)^t = \frac{(1-x)^t}{1-x^t}
       F_{r+1}(x^t,(g^{p-1})^{p^e},p), \qquad
      \IM n{q^{p^e},p} =
      \begin{cases}
        1 & n=1 \\
        (p^{1+e}-1)/t & n=t \\
        sp^e & t \mid n, n/t = (n/t)_p >1 \\
        0 & \text{otherwise}
      \end{cases}
    \end{equation*}
    for all
    $e \geq 0$ and $r \geq 0$.
  \end{lemma}
  \begin{proof}
    Note that $(q^{np^e}-1)_p$ equals  $p^e(pn/t)_p$ if $n$ is
    divisible by $t$ and $1$ if not.
   The $t$th power of the generating function $F_{r+1}(x,q^{p^e},p)$
   from Theorem~\ref{thm:Frecurpprim} is
   \begin{multline*}
     F_{r+1}(x,q^{p^e},p)^t = \exp(-t\sum_{n \geq 1} (q^{np^e}-1)_p^r
     \frac{x^n}{n})
     = \exp( -t\sum_{t \nmid n} \frac{x^n}{n} - t\sum_{t
       \mid n} p^{re}(pn)_p^r \frac{x^n}{n})\\
     = \exp( -t\sum_{t \nmid n} \frac{x^n}{n}
     - \sum_{n \geq 1} p^{re} (pn)_p^r \frac{(x^t)^n}{n})
      =
      \frac{(1-x)^t}{1-x^t} F_{r+1}(x^t,(g^{p-1})^{p^e},p)
   \end{multline*}
   An elementary calculation confirms that
    $\sum_{d \mid n} d \IM d{q,p} = (q^n-1)_p$ when the integers $\IM
    d{q,p}$ are defined as in the lemma.
  \end{proof}

  \begin{exmp} 
    The $3$-equivalence classes of prime powers prime to $3$ are
  represented by $2^{3^e}$ and $4^{3^e}$ 
  with $O(2^{3^e},3)=(2,1+e)$ and $O(4^{3^e},3)=(1,1+e)$, $e \geq
  0$ \cite[Lemma 1.11.(a)]{bmo2}  as $2$ is a primitive root modulo
  $9$. The $3$-equivalence classes of $2$, $2^3$, $4$, and $4^3$
  contain the prime powers
  \begin{alignat*}{3}
    & 2, 5, 11, 23, 29, 32, 41, 47, 59, \ldots & &\qquad &
    &8, 17, 71, 89, 125,179, 197, 233, \ldots \\
    & 4, 7, 13, 16, 25, 31, 43, 49, 61,  \ldots & &\qquad &
    & 19, 37, 64, 73, 127, 181, 199, 289, \ldots
  \end{alignat*}
  The $3$-primary generating functions satisfy
  \begin{equation*}
    F_{r+1}(x,4^{3^e},3) = \exp(- \sum_{n \geq 1} (3n)_3^r
    \frac{x^n}{n})^{3^{re}} =
    \prod_{n \geq 0} \left( \frac{(1-x^{3^n})^3}{1-(x^{3^n})^3} \right)^{3^{(r-1)(n+1)+re}}
    \quad
    F_{r+1}(x,2^{3^e},3)^2 = \frac{1-x}{1+x} F_{r+1}(x^2,4^{3^e},3)
  \end{equation*}
  according to Lemma~\ref{lemma:gpm1}, \ref{lemma:Fpexpansion}, \ref{lemma:gs}.
\end{exmp}


When $p$ is odd, $p \nmid q$ and $\ord pq$ is big, many $p$-primary 
equivariant reduced \Euc s vanish.

\begin{prop}\label{prop:vanish}
  Assume $p$ is odd, prime to $q$ and $\ord pq > 2$. Then
  $\rchi_{r+1}(n,q,p)=0$ unless $n \equiv 0,1 \bmod d$,
  $\rchi_{r+1}(n,q,p)+\rchi_{r+1}(n+1,q,p)=0$ when $n \equiv 0 \bmod
  d$, and $\rchi_{r+1}(d,q,p) = -\frac{1}{d}((q^d-1)_p^r-1)$
  where $d=\ord pq$ and $r \geq 0$.
\end{prop}
\begin{proof}
  For $n=0$, $\rchi_{r+1}(0,q,p)=1$ by convention. For $n=1$,
  $\rchi_{r+1}(1,q,p) = -(q-1)^r_p = -1$ by
  Proposition~\ref{prop:pprimr1n1}.\eqref{prop:pprimr1n1item2}. Lemma~\ref{lemma:LTE}
  shows that $(q^{md}-1)_p =(q^d-1)_pm_p$ for any $m \geq 1$ while
  $(q^k-1)_p=1$ for any $k \geq 1$ not a multiplum of $d$. By
  Corollar~\ref{cor:rpchirecur}, $-2\rchi_{r+1}(2,q,p)
  =(q-1)_p^r\rchi_{r+1}(1,q,p) + (q^2-1)^r_p\rchi_{r+1}(0,q,p) = -1+1
  = 0$. It is now clear that we can proceed by induction using Corollary~\ref{cor:rpchirecur}. 
\end{proof}


     \begin{figure}[t]
          \centering
          \addtolength{\tabcolsep}{+6pt}
          \begin{tabular}[t]{>{$}c<{$}|*{3}{>{$}c<{$}}}
  {} & (q^n-1)_2 & \IM n{q,2}  & F_{r+1}(x,q,2)\\ \hline
  q=-3 & (4n)_2 & {}
                  {\begin{cases}
                    4 & n=1 \\ 2 & n=n_2>1 \\ 0 & \text{otherwise}
                  \end{cases}}  &
                                \displaystyle \exp(-\sum\limits_{n \geq 1} (4n)_2^r
                                  \frac{x^n}{n}) \\ \hline
  q=+3^{2^e}, e>0 & 2^e(4n)_2 & 2^e \IM n{2,-3} & F_{r+1}(x,-3,2)^{2^{re}} \\ \hline
  q=3 & {}
        \begin{cases}
          2 & 2 \nmid n \\ (4n)_2 & 2 \mid n
        \end{cases} & {}
                      \begin{cases}
                         2 & n=1 \\ 3 & n=2 \\ 2 & n=n_2>2 \\ 0 & \text{otherwise}
                      \end{cases}
   &
    \displaystyle \Big(\frac{1-x}{1+x}\Big)^{2^{r-1}} F_{r+1}(x^2,-3,2)^{2^{r-1}} \\
  \hline
  q=-3^{2^e}, e>0 & {}
                 \begin{cases}
          2 & 2 \nmid n \\ 2^e(4n)_2 & 2 \mid n
        \end{cases} & {}
                    \begin{cases}
    2 & n=1 \\
    2^{2+e}-1 & n=2 \\
    2^{1+e} & n = n_2 >2 \\
    0 & \text{otherwise}
  \end{cases} &
                \displaystyle \Big(\frac{1-x}{1+x}\Big)^{2^{r-1}} F_{r+1}(x^2,3^{2^e},2)^{2^{r-1}} 
\end{tabular}
\addtolength{\tabcolsep}{-6pt}          
          \caption{$2$-primary equivariant generating functions $F_{r+1}(x,q,2)$ for $r
          \geq 0$}
          \label{fig:genfct2}
        \end{figure}

        
Next, we consider the case $p=2$.
The $2$-equivalence classes of odd prime powers are represented by
  the $2$-adic numbers $\pm 3^{2^e}$  \cite[Lemma 1.11.(b)]{bmo2} with
\begin{equation*}
          O(\pm 3^{2^e},2) =
          \begin{cases}
            (\pm 3,3) & e=0 \\ (\pm 1,3+e) & e>0
          \end{cases} 
        \end{equation*}
The $2$-classes of $-3,3^2,3,-3^2$ contain the prime powers
\begin{alignat*}{3}
  &5, 13, 29, 37, 53, 61, 101, 109, 125, \ldots &&\qquad && 9, 25, 41, 73, 89,
  121, 137, 169, 233, \ldots \\
  &3, 11, 19, 27, 43, 59, 67, 83, 107, 131, \ldots &&\qquad
  && 7, 23, 71, 103, 151, 167, 199, 263,\ldots
\end{alignat*}
        The results for $p=2$ are summarized in
        Figure~\ref{fig:genfct2}.
When $q=-3$, $((-3)^n-1)_2 = (4n)_2$ and $F_{r+1}(x,-3,2)$ given by
                                  Theorem~\ref{thm:Frecurpprim}.  An
                                  elementary calculation shows that
                                  $\sum_{d \mid n} d \IM d{-3,2} =
                                  ((-3)^n-1)_2$ when $\IM d{-3,2}$ is
                                  as in Figure~\ref{fig:genfct2}.      
When $q=3$, the rewriting
\begin{equation*}
  -\sum_{n \geq 1} (3^n-1)^r_2 \frac{x^n}{n}
  = -\sum_{2 \nmid n} 2^r \frac{x^n}{n} -\sum_{2 \mid n} (4n)_2^r\frac{x^n}{n} 
 =  -2^{r-1}\sum_{2 \nmid n} 2 \frac{x^n}{n}
 - 2^{r-1}\sum_{n \geq 1} (4n)_2^r
 \frac{(x^2)^n}{n}
\end{equation*}
translates to
\begin{equation*}
  F_{r+1}(x,3,2) 
  =\left(\frac{1-x}{1+x}\right)^{2^{r-1}} F_{r+1}(x^2,-3,2) ^{2^{r-1}} 
\end{equation*}
An elementary calculation shows that $\sum_{d \mid n} d \IM d{3,2} =(3^n-1)_2$ when $\IM d{3,2}$ is
as in Figure~\ref{fig:genfct2}. The other cases are
similar. Lemma~\ref{lemma:Fpexpansion} gives product expansions of
the $2$-primary generating functions.

\subsection{Alternative presentations of the $p$-primary equivariant reduced \Euc s}
\label{sec:recurs-equiv-reduc-pprim}

Consider the $p$-primary version of the generating function \eqref{eq:GSp},
      \begin{equation} \label{eq:Gnxqp} 
        G_n(x,q,p) = \sum_{ r \geq 0} \rchi_{r+1}(n,q,p)x^r
         = -\delta_{1,n}
  + \rchi_2(n,q,p)x + \rchi_3(n,q,p)x^2 + \cdots, \qquad n \geq 0
      \end{equation}
      where the coefficient of $x^r$ is the $(r+1)$th $p$-primary
      reduced \Euc s $\rchi_{r+1}(n,q,p)$. (Declare
      $\rchi_{r+1}(0,q,p)$ to be $1$ for all $r \geq 0$.) We have
      $G_0(x,p,q) = \frac{x}{1-x}$ and $G_1(x,q,p)=-\frac{1}{1-x(q-1)_p}$ by
      Proposition~\ref{prop:pprimr1n1}.\eqref{prop:pprimr1n1item2}.
      If $p$ is odd, $p \nmid q$ and $d=\ord pq >2$, $G_n(x,q,p)=0$
      unless $n \equiv 0,1 \bmod d$ and 
      $G_d(x,q,p) = -\frac{1}{d} \big( \frac{1}{1-x(q^d-1)_p} -
      \frac{1}{1-x} \big) = -G_{d+1}(x,q,p)$ by
      Proposition~\ref{prop:vanish}.
      The following description of
      the power series $G_n(x,q,p)$ is
      obtained exactly as in Proposition~\ref{cor:GGL} (and
      $|\lambda|$, $T(\lambda)$, and $U(\lambda,q)$ are as there).

\begin{prop}\label{prop:primGGL}
  For $r \geq 0$ and $n \geq 1$,
  \begin{equation*}
          \rchi_{r+1}(n,q,p)=
    \frac{1}{n!} 
    \sum_{ \lambda \vdash n} (-1)^{|\lambda|} T(\lambda)
    U(\lambda,q)_p^r, \qquad
        G_n(x,q,p) = \frac{1}{n!}
        \sum_{\lambda \vdash n} (-1)^{|\lambda|} \frac{T(\lambda)}{1-xU(\lambda,q)_p}
      \end{equation*}
 \end{prop}
 Examples of Proposition~\ref{prop:primGGL} with $p=2$ and
 $q = \pm 3^{2^e}$ are
\begin{alignat*}{3}  
  1!G_1(x,3^{2^e},2) &=
                         \begin{cases}
                           \frac{-1}{1-2x} & e=0 \\
                           \frac{-1}{1-2^{2+e}x} & e>0
                         \end{cases} &&\quad &
    1!G_1(x,-3^{2^e},2) &=
    \begin{cases}
      \frac{-1}{1-2^2x} & e=0 \\
      \frac{-1}{1-2x} & e>0
    \end{cases} \\
   2!G_2(x,3^{2^e},2) &=
   \begin{cases}
     \frac{-1}{1-2^3x} +\frac{1}{1-2^2x} & e = 0 \\
     \frac{-1}{1-2^{3+e}x} + \frac{1}{1-2^{4+2e}x} & e>0 
   \end{cases}  &&\quad &
   2!G_2(x,-3^{2^e},2) &=
   \begin{cases}
     \frac{-1}{1-2^3x} + \frac{1}{1-2^4x} & e=0 \\
      \frac{-1}{1-2^{3+e}x} + \frac{1}{1-2^2x} & e>0
    \end{cases} \\
  3!G_3(x,3^{2^e},2) &=
  \begin{cases}
    \frac{-2}{1-2x} + \frac{3}{1-2^4x} + \frac{-1}{1- 2^3x} & e=0\\
    \frac{-2}{1-2^{2+e}x} + \frac{3}{1-2^{5+2e}x} +
    \frac{-1}{1-2^{6+3e}x} & e >0 
  \end{cases} &&\quad &
  3!G_3(x,-3^{2^e},2) &=
  \begin{cases}
    \frac{-2}{1-2^2x} + \frac{3}{1-2^5x} + \frac{-1}{1-2^6x} & e=0 \\
    \frac{-2}{1-2x} + \frac{3}{1-2^{4+e}x} + \frac{-1}{1-2^3x} & e>0
  \end{cases}
\end{alignat*}


Define the reciprocal $p$-primary equivariant reduced \Euc ,
$\rchi^{-1}_{r+1}(n,q,p)$,  to be
the coefficient of $x^n$ in the reciprocal of $F_{r+1}(x,q,p)$. Then
\begin{align}
    &1 + \sum_{n \geq 1} \rchi^{-1}_{r+1}(n,q,p) x^n =
      F_{r+1}(x,q,p)^{-1} \label{eq:FRpqx-1}   \\
    &1 + \sum_{r \geq 1} \rchi^{-1}_{r+1}(n,q,p) x^r = \frac{1}{n!}  
    \sum_{\lambda \vdash n} \frac{T(\lambda)}{1-xU(\lambda,q)_p} \qquad
     \rchi^{-1}_{r+1}(n,q,p)  = \frac{1}{n!}\sum_{ \lambda \vdash n}  T(\lambda)
    U(\lambda,q)_p^r 
\end{align}  
For instance, the multiplicative order $\ord 52=4$ and the sequences
of (reciprocal) third $5$-primary reduced equivariant \Euc s
\begin{align*} 
  &(\rchi_3(n,2,5))_{n \geq 0} = ( 1, -1, 0, 0, -6, 6, 0, 0, 15, -15,
    0, 0, -20, 20, 0, 0, 15, -15, 0, 0, -36 \cdots ) \\
  &(\rchi_3^{-1}(n,2,5))_{n \geq 0} = (  1, 1, 1, 1, 7, 7, 7, 7, 28,
    28, 28, 28, 84, 84, 84, 84, 210, 210, 210, 210, 492, \ldots) 
\end{align*}
illustrate Proposition~\ref{prop:vanish} and Definition~\eqref{eq:FRpqx-1}.

We noted in Section~\ref{sec:altern-pres-equiv} that
$\rchi^{-1}_2(n,q) = q^{n-1}(q-1)$ counts semi-simple
classes in $\GL n{\F_q}$. The following corollary is the $p$-primary
analogue.
 

       \begin{cor}\label{cor:F2p}
         The coefficient of $x^n$ in the power series 
         \begin{equation*}
           F_2(x,q,p)^{-1} =T_{\IM {}{q,p}} (1-x)^{-1} =
              \exp\big(\sum_{n \geq 1} (q^n-1)_p \frac{x^n}{n} \big) =
              \prod_{d \geq 1}(1-x^d)^{-\IM d{q,p}}
            \end{equation*}
            is the number of $p$-singular semi-simple classes in
            $\GL n{\F_q}$. 
       \end{cor}

The recursive relations
$F_{r+1}(x,q,p)^{\pm 1} = T_{\IM{}{q,p}}
F_r(x,q,p)^{\pm 1} = T_{-\IM{}{q,p}}
F_r(x,q,p)^{\mp 1}$ give a sequence of $p$-primary \pol\ identities
\begin{equation*}  
  \rchi_{r+1}^{-\varepsilon_1}(n,q,p) =
  \sum_{\lambda \in M_n}
   \prod_{\{ d \mid \exists m \colon
     (m,d) \in B(\lambda) \}}
  \binom{ \varepsilon_1 \varepsilon_2 \IM d{q,p}}{[E(\lambda,(m,d)) \mid  (m,d) \in B(\lambda)]}
   \prod_{\{m \mid (m,d) \in B(\lambda)\}}
   \rchi_r^{-\varepsilon_2}(m,q^d,p), \qquad r \geq 1
\end{equation*}
where $\rchi_{r+1}^{-\varepsilon_1}(n,q,p) =
  \frac{1}{n!} \sum_{\lambda \vdash n}
  \varepsilon_1^{|\lambda|}T(\lambda)U(\lambda,q)_p^{r}$,
 $\rchi_r^{-\varepsilon_2}(m,q^d,p) = \sum_{\mu \vdash
   m}\varepsilon_2^{|\mu|}T(\mu)U(\mu,q^d)_p^{r-1}$ for $\varepsilon_1, \varepsilon_2 = \pm 1$.
The $r=1$ term in this sequence, a $p$-primary version of
\cite[Theorem A, B]{thevenaz92poly}, is 
 \begin{equation*}
   \rchi_2^{-\varepsilon_1}(n,q,p) =
     \frac{1}{n!}
    \sum_{\lambda \vdash n}
    \varepsilon_1^{|\lambda|}T(\lambda)U(\lambda,q)_p =
     \begin{cases}
      \displaystyle \sum_{\lambda \vdash n} (-1)^{|\lambda|} \prod_{d \in
    B(\lambda)} \binom{-\varepsilon_1 \IM d{q,p}}{E(\lambda,d)} &
  \varepsilon_2=-1\\
  \displaystyle \sum_{\lambda \in M_n} \prod_d \binom{\varepsilon_1\IM
                                                                 d{q,p}}{[E(\lambda,(m,d))
                                                                 \mid
                                                                 (m,d)\in
                                                                 B(\lambda)]}  &
  \varepsilon_2=+1
                                                             \end{cases}
   \end{equation*}
where we used that $\rchi_1(m,q^d,p)=-\delta_{1,m}$
(Proposition~\ref{prop:pprimr1n1}.\eqref{prop:pprimr1n11}) contributes only for
$m=1$ and
$\rchi^{-1}_1(m,q^d,p)=1$ for all $m \geq 1$.

When $q$ is a power of $p$, the identity is the only $p$-singular
semi-simple element in $\GL n{\F_q}$ and the generating function of
Corollary~\ref{cor:F2p} is
$F_2(x,q,p)^{-1} = (1-x)^{-1} = 1 + \sum_{n \geq 1} x^n$. When
$p \nmid q$, all $p$-singular classes are semi-simple so
     \begin{gather*}
          1+ \sum_{n \geq 1}  |\GL n{\F_q}_p/\GL n{\F_q}|  x^n =F_2(x,q,p)^{-1} 
        \end{gather*}
        where $|\GL n{\F_q}_p/\GL n{\F_q}|$ is the number of $p$-singular classes.
For instance, the groups $\GL 2{\F_q}$ and $\GL 3{\F_q}$ contain
       \begin{align*}
         \rchi_2^{-1}(2,q,p) &=
         \frac{1}{2!}((q^2-1)_p + (q-1)^2_p) \stackrel{\varepsilon_2=-1}{=}
         \binom{\IM 1{q,p}}{2} - \binom{-\IM 2{q,p}}{1}                                     
         \\ &\stackrel{\varepsilon_2=+1}{=}
         \binom{\IM 2{q,p}}{1} + \binom{\IM 1{q,p}}{1} + \binom{\IM
             1{q,p}}{2} \\
         \rchi_2^{-1}(3,q,p) &=
         \frac{1}{3!}(2(q^3-1)_p+3(q-1)_p(q^2-1)_p+(q-1)^3_p) \\&\stackrel{\varepsilon_2=-1}{=}
          -\binom{-\IM 3{q,p}}{1} + \binom{-\IM 2{q,p}}{1} \binom{-\IM 1{q,p}}{1}                                                                                       -\binom{-\IM 1{q,p}}{3}                                                                       
         \\ &\stackrel{\varepsilon_2=+1}{=}
         \binom{\IM 3{q,p}}{1} + \binom{\IM 1{q,p}}{1} + \binom{\IM
           1{q,p}}{1} \binom{\IM 2{q,p}}{1} + \binom{\IM 1{q,p}}{1,1}
         + \binom{\IM 1{q,p}}{3} 
       \end{align*}
         $p$-singular conjugacy classes when $p \nmid q$.

\section*{Acknowledgments}
\label{Acknowledgments}
I warmly thank the participants in a discussion thread at the internet
site MathOverflow \cite{ofir2016} for some extremely helpful hints and
two anonymous referees for a dramatic shortening and improvement of
the original version of this article.  I used the computer algebra
system Magma \cite{magma} for concrete and experimental computations
and the On-Line Encyclopedia of Integer Sequences for reference.

\def\cprime{$'$} \def\cprime{$'$} \def\cprime{$'$} \def\cprime{$'$}
  \def\cprime{$'$}
\providecommand{\bysame}{\leavevmode\hbox to3em{\hrulefill}\thinspace}
\providecommand{\MR}{\relax\ifhmode\unskip\space\fi MR }
\providecommand{\MRhref}[2]{%
  \href{http://www.ams.org/mathscinet-getitem?mr=#1}{#2}
}
\providecommand{\href}[2]{#2}


\begin{thebibliography}{10}

\bibitem{atiyah&segal89}
Michael Atiyah and Graeme Segal, \emph{On equivariant {E}uler characteristics},
  J. Geom. Phys. \textbf{6} (1989), no.~4, 671--677. \MR{1076708 (92c:19005)}

\bibitem{bona2006}
Mikl\'{o}s B\'{o}na, \emph{A walk through combinatorics}, second ed., World
  Scientific Publishing Co. Pte. Ltd., Hackensack, NJ, 2006, An introduction to
  enumeration and graph theory, With a foreword by Richard Stanley.
  \MR{2361255}

\bibitem{magma}
Wieb Bosma, John Cannon, and Catherine Playoust, \emph{The {M}agma algebra
  system. {I}. {T}he user language}, J. Symbolic Comput. \textbf{24} (1997),
  no.~3-4, 235--265, Computational algebra and number theory (London, 1993).
  \MR{1484478}

\bibitem{bmo2}
C.~{Broto}, J.~M. {M{\o}ller}, and B.~{Oliver}, \emph{{Automorphisms of fusion
  systems of finite simple groups of Lie type}}, ArXiv e-prints (2016), Mem.
  Amer. Math. Soc. (to appear).

\bibitem{bmo1}
Carles Broto, Jesper~M. M{\o}ller, and Bob Oliver, \emph{Equivalences between
  fusion systems of finite groups of {L}ie type}, J. Amer. Math. Soc.
  \textbf{25} (2012), no.~1, 1--20. \MR{2833477}

\bibitem{brown82}
Kenneth~S. Brown, \emph{Cohomology of groups}, Graduate Texts in Mathematics,
  vol.~87, Springer-Verlag, New York, 1982. \MR{83k:20002}

\bibitem{cr}
Charles~W. Curtis and Irving Reiner, \emph{Representation theory of finite
  groups and associative algebras}, AMS Chelsea Publishing, Providence, RI,
  2006, Reprint of the 1962 original. \MR{2215618 (2006m:16001)}

\bibitem{green55}
J.~A. Green, \emph{The characters of the finite general linear groups}, Trans.
  Amer. Math. Soc. \textbf{80} (1955), 402--447. \MR{0072878}

\bibitem{HKR2000}
Michael~J. Hopkins, Nicholas~J. Kuhn, and Douglas~C. Ravenel, \emph{Generalized
  group characters and complex oriented cohomology theories}, J. Amer. Math.
  Soc. \textbf{13} (2000), no.~3, 553--594 (electronic). \MR{1758754}

\bibitem{ofir2016}
Ofir~Gorodetsky (https://mathoverflow.net/users/31469/ofir gorodetsky),
  \emph{Transforming numbers of irreducible polynomials}, MathOverflow,
  URL:https://mathoverflow.net/q/252800 (version: 2016-10-22).

\bibitem{irelandrosen90}
Kenneth Ireland and Michael Rosen, \emph{A classical introduction to modern
  number theory}, second ed., Graduate Texts in Mathematics, vol.~84,
  Springer-Verlag, New York, 1990. \MR{1070716}

\bibitem{knorr_robinson:89}
Reinhard Kn{\"o}rr and Geoffrey~R. Robinson, \emph{Some remarks on a conjecture
  of {A}lperin}, J. London Math. Soc. (2) \textbf{39} (1989), no.~1, 48--60.
  \MR{989918 (90k:20020)}

\bibitem{lidlnieder97}
Rudolf Lidl and Harald Niederreiter, \emph{Finite fields}, second ed.,
  Encyclopedia of Mathematics and its Applications, vol.~20, Cambridge
  University Press, Cambridge, 1997, With a foreword by P. M. Cohn.
  \MR{1429394}

\bibitem{jmm:partposet2017}
Jesper~M. M{\o}ller, \emph{Equivariant {E}uler characteristics of partition
  posets}, European J. Combin. \textbf{61} (2017), 1--24. \MR{3588706}

\bibitem{quillen78}
Daniel Quillen, \emph{Homotopy properties of the poset of nontrivial
  {$p$}-subgroups of a group}, Adv. in Math. \textbf{28} (1978), no.~2,
  101--128. \MR{MR493916 (80k:20049)}

\bibitem{rosen2002}
Michael Rosen, \emph{Number theory in function fields}, Graduate Texts in
  Mathematics, vol. 210, Springer-Verlag, New York, 2002. \MR{1876657}

\bibitem{sagan:symmetric}
Bruce~E. Sagan, \emph{The symmetric group}, second ed., Graduate Texts in
  Mathematics, vol. 203, Springer-Verlag, New York, 2001, Representations,
  combinatorial algorithms, and symmetric functions. \MR{1824028}

\bibitem{serre73}
Jean-Pierre Serre, \emph{A course in arithmetic}, Springer-Verlag, New York,
  1973, Translated from the French, Graduate Texts in Mathematics, No. 7.
  \MR{MR0344216 (49 \#8956)}

\bibitem{stanley97}
Richard~P. Stanley, \emph{Enumerative combinatorics. {V}ol. 1}, Cambridge
  Studies in Advanced Mathematics, vol.~49, Cambridge University Press,
  Cambridge, 1997, With a foreword by Gian-Carlo Rota, Corrected reprint of the
  1986 original. \MR{MR1442260 (98a:05001)}

\bibitem{stanley99}
\bysame, \emph{Enumerative combinatorics. {V}ol. 2}, Cambridge Studies in
  Advanced Mathematics, vol.~62, Cambridge University Press, Cambridge, 1999,
  With a foreword by Gian-Carlo Rota and appendix 1 by Sergey Fomin.
  \MR{1676282 (2000k:05026)}

\bibitem{tamanoi2001}
Hirotaka Tamanoi, \emph{Generalized orbifold {E}uler characteristic of
  symmetric products and equivariant {M}orava {$K$}-theory}, Algebr. Geom.
  Topol. \textbf{1} (2001), 115--141 (electronic). \MR{1805937}

\bibitem{thevenaz92poly}
Jacques Th{\'e}venaz, \emph{Polynomial identities for partitions}, European J.
  Combin. \textbf{13} (1992), no.~2, 127--139. \MR{1158806 (93j:11069)}

\bibitem{thevenaz93Alperin}
\bysame, \emph{Equivariant {$K$}-theory and {A}lperin's conjecture}, J. Pure
  Appl. Algebra \textbf{85} (1993), no.~2, 185--202. \MR{1207508 (94c:20022)}

\bibitem{dieck}
Tammo tom Dieck, \emph{Transformation groups}, de Gruyter Studies in
  Mathematics, vol.~8, Walter de Gruyter \& Co., Berlin, 1987. \MR{MR889050
  (89c:57048)}

\bibitem{webb87}
P.~J. Webb, \emph{A local method in group cohomology}, Comment. Math. Helv.
  \textbf{62} (1987), no.~1, 135--167. \MR{882969 (88h:20065)}

\end{thebibliography}
\end{document}